\documentclass[12pt,oneside]{amsart}

\usepackage{persomath}
\usepackage{persobase_patched}

\AtBeginDocument{\selectlanguage{english}}

\calclayout

\usepackage{nicefrac,setspace}
\usepackage{xcolor}
\usepackage{bm}
\usepackage{booktabs}
\usepackage{microtype}

\newcommand{\Spec}{{\rm Spec}}

\newtheorem{defn}[defi]{Definition}

\newtheorem*{theorem*}{Theorem}
\newtheorem*{lemma*}{Lemma}
\newtheorem*{corollary*}{Corollary}

\newcommand{\showcomments}{no}
\renewcommand{\showcomments}{yes}

\newsavebox{\commentbox}
{\ifthenelse{\equal{\showcomments}{yes}}%
{\footnotemark
    \begin{lrbox}{\commentbox}
    \begin{minipage}[t]{1.25in}\raggedright\sffamily\tiny
    \footnotemark[\arabic{footnote}]}
{\begin{lrbox}{\commentbox}}}%
{\ifthenelse{\equal{\showcomments}{yes}}%
{\end{minipage}\end{lrbox}\marginpar{\usebox{\commentbox}}}
{\end{lrbox}}}

\newenvironment{problem}{\begin{prob}}{\end{prob}}

\date{\today}

\thanks{
Rémi Coulon acknowledges support from the Agence Nationale de la Recherche under the Grant \emph{GoFR} ANR-22-CE40-0004 and from the Région Bourgogne-Franche-Comté under the grant \emph{ANER 2024 GGD}.
His institute receives support from the EIPHI Graduate School (contract ANR-17-EURE-0002).
Daniel T. Wise's research is supported by the ISF}

\title[Growth spectrum of hyperbolic groups]{On the growth spectrum of hyperbolic groups}

\author{R\'emi Coulon}
\address{Universit\'e Bourgogne Europe, CNRS, IMB UMR 5584, 21000 Dijon, France}
\email{remi.coulon@cnrs.fr}

\author{Michail Louvaris}
\address{Dept of Math, Yale, New Haven, USA}
\email{michail.louvaris@yale.edu}

\author{Daniel T. Wise}
\address{Dept. of Math, Weizmann Institute, Rehovot, Israel}
\email{daniel.wise@weizmann.ac.il}

\author{Gal Yehuda}
\address{Dept of Math, Yale, New Haven, USA}
\email{gal.yehuda@yale.edu}

\begin{document}

\begin{abstract}
We study the growth spectrum of groups acting on hyperbolic spaces, i.e.\ the set of exponential growth rates achieved by subgroups.
For a finitely generated free group or a surface group acting convex-cocompactly on a proper geodesic hyperbolic metric space, we prove that the growth spectrum is the full interval $[0, \omega_G]$.
For any hyperbolic group, we prove that the growth spectrum contains a large interval $[0, \omega_{\mathcal{F}}]$ where $\omega_{\mathcal{F}} \geq \omega_G / 2$, with strict inequality when the action is divergent.
In the case of the Cayley graph of a free group, we also present an approach via the non-backtracking matrix of the configuration model, connecting the density of growth rates to a spectral concentration result for random graphs.
\end{abstract}

\maketitle

\tableofcontents

% -------------------------------------------------------
\section{Introduction}
% -------------------------------------------------------

% \begin{com}
% 	RC. I revised the motivations.
% 	We indeed focus on the growth rate, not on the growth function. 
% 	There are plenty enough good reasons to look at growth rates.
% \end{com}

%For a finitely generated group $G$ and a set of generators $S$ with Cayley graph $\Upsilon = \Upsilon_{G,S}$, let $f(n)$ denote the number of elements of length $\leq n$.
%The \emph{exponential growth rate} of $G$ with respect to $\Upsilon$ is defined as
%$$\omega(G, \Upsilon) = \lim_{n\rightarrow \infty}\frac{1}{n}\ln f(n).$$
Growth in groups is a topic that has driven many developments in geometric group theory with striking achievements such as Gromov's characterization of groups with polynomial growth or Grigorchuk's examples of groups with intermediate growth.
In this article we are interested in groups with exponential growth.
Let $(X, \distV)$ be a metric space and $G$ a group acting properly, by isometries on $X$.
% \begin{com}
% RC. I removed the preliminary definition of a Cayley graph.
% It is indeed verbatim the same.
% There is no need to give twice the same definition.
% \end{com}
The \emph{exponential growth rate of $G$ with respect to $X$}, which measures the ``size'' of its orbits, is defined as
\begin{equation*}
	\omega(G,X) = \limsup_{r \to \infty} \frac 1r \ln \card{\set{g \in G}{\dist x{gx} \leq r}},
\end{equation*}
where $x$ is a point in $X$.
Alternatively, it is the critical exponent of the Poincar\'e series
\begin{equation*}
	\mathcal P_G(s,x) = \sum_{g \in G} e^{-s\dist x{gx}}.
\end{equation*}
This quantity does not depend on $x \in X$, but does depend on the space $X$.
Nevertheless, if there is no ambiguity, we simply denote it by $\omega_G$.
For instance, the rank~$r$ free group $F_r$ has growth rate $\omega( F_r, X) = \ln(2r-1)$ when $X$ is a Cayley graph corresponding to a free generating set.
%Finitely generated groups with polynomial growth were famously characterized by Gromov as those which are virtually nilpotent.
%Since non-elementary word-hyperbolic groups contain a copy of a rank~2 free group, they have exponential growth-rates, i.e.\ $\omega_G>0$ for any geometric action on a metric space $X$, and studies have been made of their growth-rates and the growth of their conjugacy classes, with a main result being that $ce^{\omega n} < f(n) < Ce^{\omega n}$ in this case \cite{CoornaertKnieper2002}.
%There are various surprises about the nature of the growth function $f$; for instance, it is a rational function when $G$ is word-hyperbolic or more generally a geodesically automatic group \cite{epstein1992word}.
If $G$ is a torsion-free discrete convex cocompact group of isometries of $X = \mathbb H^{n+1}$, then $\omega(G, X)$ has numerous interpretations that make it a central object in Riemannian geometry:
not only does it measure the asymptotic behavior of the orbits of $G$, it is also the entropy of the geodesic flow on the unit tangent bundle of the quotient manifold $M = X / G$, the Hausdorff dimension of the limit set $\Lambda(G) \subseteq \partial X$ of $G$, it is also related to the bottom of the spectrum of the Laplace-Beltrami operator on $M$, etc.

Given a collection $\mathcal H$ of subgroups of $G$, its \emph{growth spectrum} is
\begin{equation*}
	{\rm Spec}(\mathcal H,X) = \set{\omega(H, X)}{H \in \mathcal H}.
\end{equation*}
This set is contained in $\intval 0{\omega_G}$.
When $\mathcal H$ is the collection of all subgroups of $G$, we simply denote it by ${\rm Spec}(G,X)$.
There are interesting relationships between the growth rate of $G$ and growth rates of its subgroups.
Of course a finite index subgroup $G'$ of $G$ has the same growth-rate as $G$, so the focus is on infinite index subgroups of $G$.

A classical result on this topic is given by Corlette \cite{Corlette:1990br}:
if $G$ is a lattice in quaternionic or octonionic hyperbolic space then $G$ has a \emph{growth-gap} in the sense that there exists $\omega< \omega_G$ such that $\omega(H,X)\leq \omega$ for each infinite index subgroup $H\subset G$.
This has been generalized to a group $G$ with \emph{Property~$(T)$}, which means that every isometric affine action on a Hilbert space has a global fixed point.
Inspired by the work of Dougall--Sharp \cite{dougall2016amenability}, it was shown in \cite{Coulon:2018aa, coulonTwistedPattersonSullivanMeasures2025} that every hyperbolic group with Property~$(T)$ has a growth-gap. 

For a hyperbolic group, $\omega(H,X)< \omega_G$ whenever $H$ is an infinite index quasiconvex subgroup of a hyperbolic group $G$ \cite{DahmaniFuterWise2019}.
% \begin{com}
% 	RC. Removed the definition of quasi-convex subgroup.
% 	I think this is a sufficiently well known concept in GGT.
% \end{com}
However, given a finitely generated free group $F$, there is a sequence $(H_n)$ of infinite index finitely generated subgroups such that $\omega(H_n,X)$ converges to $\omega_F$. 
Hence $F$ does not have a growth gap.
This no-growth-gap result was generalized from free groups to fundamental groups of compact special cube complexes in \cite{LiWise2020}. 

\subsection{Main results}

A natural problem is to describe ${\rm Spec}(G,X)$: does it cover the whole interval $\intval 0{\omega_G}$?
In this paper we prove several results on the growth spectrum of free groups, surface groups, and more generally, hyperbolic groups.

\begin{mthm}
\label{res: main new}
	Let $G$ be a finitely generated free group or a closed surface group acting properly, by isometries on a proper, geodesic, hyperbolic metric space $(X,d)$.
	Suppose that this action is convex-cocompact.
	Then ${\rm Spec}(G,X) = \intval 0{\omega_G}$.
	Moreover, if $\mathcal F$ stands for the collection of all finitely generated free subgroups of $G$, then ${\rm Spec}(\mathcal F, X)$ is dense in $\intval 0{\omega_G}$.
\end{mthm}

In the special case where $G = F_r$ is a free group and $X$ is its Cayley graph with respect to a free basis, \autoref{res: main new} states that the set of growth-rates of finitely generated subgroups of $F_r$ is dense in $[0, \ln(2r-1)]$.
We present two proofs of the density part of \autoref{res: main new} for free groups.
The first, given in \autoref{sec:eigenvalue_approach}, uses the non-backtracking matrix of a model for random graphs and relies on a probabilistic concentration result (\autoref{thm:probabilistic_claim}).
The second, given in \autoref{sec:density_free_group}, is purely geometric and also works for any convex-cocompact action on a hyperbolic space.

\medskip
In view of Corlette's result it is natural to ask whether there are groups with \emph{several} gaps in their growth spectrum.
Although we do not fully answer this question, we prove that there is no growth gap in the lower half of the growth spectrum of hyperbolic groups.
Recall that the action of a group $G$ on a metric space $X$ is \emph{divergent} if its Poincar\'e series $\mathcal P_G(s,x)$ diverges at $s = \omega_G$ for some (hence every) $x \in X$.

\begin{mthm}
\label{res: free subgroup spectrum}
	Let $(X,d)$ be a proper, geodesic, hyperbolic space.
	Let $G$ be a group acting properly, by isometries on $X$.
	Denote by $\mathcal F$ the collection of all free subgroups of $G$ which are convex-cocompact (for the action on $X$).
	There is $\omega_{\mathcal F} \in \R_+$ such that 
	\begin{enumerate}
		\item \label{enu: free subgroup spectrum - density}
		${\rm Spec}(\mathcal F,X)$ is dense in $\intval 0{\omega_{\mathcal F}}$;
		\item \label{enu: free subgroup spectrum - realized}
		$\intval 0{\omega_{\mathcal F}}$ is contained in ${\rm Spec}(G,X)$,
		\item \label{enu: free subgroup spectrum - 1/2}
		$\omega_\mathcal F \geq \omega_G / 2$,  with a strict inequality if the action of $G$ is divergent.
	\end{enumerate}	
\end{mthm}

\subsection{The eigenvalue approach for the Cayley graph}\label{sub:eigenvalue_approach_intro}

The first approach to the density part of \autoref{res: main new} in the case of a free group acting on its Cayley graph (with respect to a free basis) is via the spectral theory of the non-backtracking matrix.
The free group $F$ of rank $r$ is identified with $\pi_1 B_r$ where $B_r$ is a bouquet of $r$ circles.
The universal cover of $B_r$ is a tree $X$  isomorphic to the Cayley graph of $F$ with respect to the basis of $F$ given by the loops of $B_r$.
Subgroups of $F$ are studied via \emph{immersions} $A\rightarrow B_r$, which are locally injective maps \cite{Stallings83}.
Every finitely generated subgroup $H\subset F$ arises as the $\pi_1$-image of such an immersion with compact domain.

For each $\omega\in [0,\ln(2r-1)]$ it is easy to imagine subtrees $Y$ of $X$ whose growth-rate is $\omega$: one simply chooses $Y$ so that the average degrees of sufficiently large balls approach $e^\omega+1$.
The challenge is to realize such subtrees as universal covers of compact graphs that admit immersions into $B_r$.
By Leighton's theorem \cite{AngluinGardiner1981}, every finite graph of degree at most $2r$ admits an immersion into $B_r$, so the problem reduces to finding graphs with the right growth.

The key tool is the \emph{non-backtracking matrix} $M_\Gamma$ of a graph $\Gamma$, a variation of the adjacency matrix of $\Gamma$.
It is a matrix indexed by oriented edges whose power $M_\Gamma^k$ ``counts'' the number of non-backtracking paths of length $k$ in $\Gamma$.
Therefore its leading eigenvalue $\lambda_1(M_\Gamma)$ satisfies the following property :  $\ln \lambda_1(M_\Gamma)$ equals $\omega(H,Y)$, where $H = \pi_1\Gamma$ and $Y$ is the universal cover of $\Gamma$ \cite{angel2015non}.
An immersion $\Gamma \to B_r$ induces an equivariant isometry from $Y$ to $X$.
Hence $\ln \lambda_1(M_\Gamma)$ is also the growth rate of $H$ seen as subgroup of $F$.
Using a probabilistic concentration result -- see \autoref {thm:probabilistic_claim} proved in \cite{louvaris2024density} -- one can produce graphs $\Gamma$ with a precise control on the leading eigenvalue of their non-backtracking matrix, hence subgroups with dense growth rates in $[0, \omega_F]$.
The details of this strategy are sketched in \autoref{sec:eigenvalue_approach}.
Once the density part is established, the continuity part of \autoref{res: main new} for the free group $F$ acting on $X$ is also proven \cite{louvarisSubgroupsFreeGroup2025}. 

\subsection{The geometric approach}
\label{sub:geometric_approach_intro}
The second approach is purely geometric.
It starts with the same idea though.
Assume for simplicity that $X$ is still the Cayley graph of the free group $F$ with respect to a free basis.
Each finitely generated subgroup of $F$ is represented by the immersion of a finite graph in $B_r$.
Consider now a $2r$-regular graph $\Gamma_0$ with a very large girth $N$.
Since $\Gamma_0$ is regular, any subgroup $H_0$ of $F$ obtained from an immersion of $\Gamma_0$ in $B_r$ will have finite index, and thus $\omega_{H_0} = \omega_F$.
Now produce by induction a sequence of graphs $(\Gamma_n)$ where $\Gamma_{n+1}$ is obtained from $\Gamma_n$ by subdividing one edge.
After choosing immersions in $B_r$, it provides a sequence of finitely generated subgroups $(H_n)$ of $F$.
Since the original group $\Gamma_0$ had a large girth, we can prove that at each step the growth rate of $H_n$ decreases by at most by $\omega_F/N$.
If in the process we have inductively subdivided sufficiently many times \emph{every} edge of the original graph $\Gamma_0$, then the resulting group $H_n$ can have an arbitrary small growth rate.
In this way the set $\{\omega(H_n), n \in \N\}$ approximates, up to an error $1/N$ every number in $[0, \omega_F]$.
Since the girth $N$ of $\Gamma_0$ can be chosen arbitrarily large, this gives arbitrarily fine control over the growth rates, yielding the density part of \autoref{res: main new} without any probabilistic tool.
This argument can be extended to the general situation where $X$ is a hyperbolic geodesic metric space, with a convex-cocompact action of the group $F$.

In this geometric setting, the continuity part of \autoref{res: main new} for the free group $F$ can be obtained by combining a ping-pong argument with the lower semi-continuity of the map assigning to each subgroup its growth rate.
% \begin{com}
% 	RC. I expanded a bit that part, since it is the core of the paper!
% \end{com}

\subsection{Open problems}

Given a group $G$ acting on a metric space $X$ we say that the \emph{growth density} holds if the set of growth rates of finitely generated subgroups of $G$ is dense in $[0, \omega_G]$.
% \begin{com}
% 	RC. I shorten the name of the property to \emph{growth density}.
% 	Indeed the terminology was not always consistent in the questions below.
% \end{com}

\begin{problem}
Does growth density hold for $A*B$ if it holds for $A$ and $B$?
What properties on $A,B$ would enable this?
\end{problem}

\begin{problem}
Does growth density hold for $\pi_1X$ when $X$ is compact special and $\pi_1X$ is non-elementary (relatively) hyperbolic?
\end{problem}

\begin{problem}
Does growth density hold for $\pi_1M^n$ for any classical hyperbolic $n$-manifold with $n\geq 3$?
\end{problem}

Recall that when $G$ is a word-hyperbolic group with Property~$(T)$, there is a gap $(\omega_G-\epsilon, \omega_G)$ at the top of the range of possible growth exponents of subgroups.

\begin{problem}
Are there finitely many ``gaps'' in the range of growths within $[0,\omega_G]$?
Can one produce further gaps with a short exact sequence?
\end{problem}

\begin{problem}[Foundational]
Is the existence of a growth gap (at the top) independent of generators?
Is growth density independent of generators?
Is it independent of choice of proper cocompact action on metric space $X$, i.e.\ independent of choice of proper metric on $G$?
\end{problem}

\begin{rema}
	If $F$ is the free group acting on its Cayley graph with respect to a free basis, the density part of \autoref{res: main new} was first proved by the last three authors in  \cite{louvaris2024density} using a probabilistic approach.
	Since then, a deterministic strategy has been independently proposed in \cite{Coulon:2025cr} and \cite{timar2026density}.
	The goal of this article is to extend this method to a broader geometric context.
	% \begin{com}
	% 	RC. I rephrased the remark to make sure that nobody will tell us that we ``stole'' the result from Timar.
	% \end{com}
\end{rema}

% -------------------------------------------------------
\section{The eigenvalue approach for the Cayley graph}
\label{sec:eigenvalue_approach}
% -------------------------------------------------------

\begin{nota}
	% \begin{com}
	% 	RC. To be consistent throughout the paper, I rewrote everything using Serre's conventions for graphs.
	% \end{com}
	In this article we use the definition of graphs given by Serre \cite{Serre:1977wy}.
	More precisely a graph is a pair $\Gamma = (V,E)$ where $V$ and $E$ are respectively the vertex and the edge set.
	It comes with an involution of $E$ reversing the orientation of edges, which we write $e \mapsto \bar{e}$.
	The initial and terminal vertices of an edge $e \in E$ are denoted by $o(e)$ and $t(e)$ respectively.
	The degree of a vertex $v \in V$, denoted by $\deg (v)$ is the number of edges $e \in E$ such that $o(e) = v$.
	A loop is an edge $e \in E$ such that $o(e) = t(e)$.
	We often confuse the graph $\Gamma$ and its topological realization.
\end{nota}

In this section we prove the density part of \autoref{res: main new} in the special case where $G = F_r$ is a free group and $X$ is its Cayley graph with respect to a free basis, via the non-backtracking matrix and the configuration model.
%In this section we use the combinatorial convention for graphs: a graph $G = (V, E)$ consists of a vertex set $V$ and an edge set $E$ of unordered pairs $\{u,v\}$ with $u \neq v$.
%This differs from Serre's convention used in \autoref{sec:density_free_group}, where $E$ is a set of oriented edges equipped with an involution; the two conventions are equivalent up to the identification of each unordered edge $\{u,v\}$ with the pair of oriented edges $(u,v)$ and $(v,u)$.

\subsection{Subgroups, subtrees, and immersed graphs}\label{sub:groups_to_graphs}

Let $B_r$ be a graph with one vertex and $r$ edges.
The fundamental group of $B_r$ based at its vertex is isomorphic to $F_r$.
Letting the basepoint of $B_r$ be its vertex, we have $\pi_1B_r\cong F_r$
where $F_r$ is a free group of rank~$r$ whose basis corresponds to the directed loops of $B_r$.
An \emph{immersion} $f: A\rightarrow B$ of graphs is a locally-injective map sending vertices to vertices and edges to edges.
The main method used to study free groups and their subgroups
was popularized by Stallings \cite{Stallings83}. In particular, we use:

\begin{lemm}[See Stallings {\cite[prop~5.3]{Stallings83}}]
\label{lem:immersion pi_1 injective} 
	Let $(A,a)$ and $(B,b)$ be graphs
	with basepoints $a$ and $b$.
	Let $f: A\rightarrow B$ be a basepoint preserving immersion. Then $f_*: \pi_1(A,a)\rightarrow \pi_1(B,b)$ is injective.
\end{lemm}

% \begin{com}
% 	RC. I suggest that we remove the proof and give the precise reference to Stallings' paper.
% 	The given proof is essentially the same.
% \end{com}

%The proof of 
%\autoref{lem:immersion pi_1 injective}  closest to the intentions of this text is that the induced map
%$\tilde f: \widetilde A \rightarrow \widetilde B$ between the universal covers is injective, since it is an immersion between trees. This implies $\pi_1$-injectivity as follows:
%If $\sigma \rightarrow A$ is a closed based path
%with $[\sigma]\neq 1_{\pi_1A}$ then its lift $\tilde \sigma \rightarrow \widetilde A$ is not closed. Hence $\tilde f \circ \tilde \sigma$ is not closed in $\widetilde B$ by injectivity of $\tilde f$. Hence $[f\circ\sigma] \neq 1_{\pi_1B}$.

\begin{lemm}
    Let $\Gamma = (V, E)$ be a finite graph whose degrees are in $\{1, \ldots, 2r\}$. 
 %    \begin{com}
 %    	RC. The proof is only for finite graphs. So the statement should say finite. 
	% There is probably already a reference in the literature for this (even for infinite graphs).
 %    \end{com}
    % \begin{com}
    % 	RC. For consistency, and to avoid confusions, I renamed the graphs $\Gamma$ instead of $G$.
    % \end{com}
    Then there is an immersion from $\Gamma$ to $B_r$. 
\end{lemm}

\begin{proof}
	According to Konig \cite[Chapter~XI, thm~6]{Konig:1990th}, there is a partition of $E$ into $r$ (possibly empty) subsets $E_1, \dots, E_r$ with the following properties.
	For each $i \in \{1, \dots, r\}$, the set $E_i$ is invariant under the involution $e \mapsto \bar e$.
	Moreover, for every vertex $v \in V$, for every $i \in \{1, \dots, r\}$, there are at most two edges $e \in E_i$ such that $o(e) = v$.
	It follows that the subgraph induced by $E_i$ is a union of disjoint cycles and segments.
	We choose of each for this component an arbitrary orientation.
	It defines an immersion $\Gamma \to B_r$. 
\end{proof}

\begin{coro} \label{cor:immersion_and_subtree}
$\widetilde \Gamma \rightarrow \widetilde B_r$ is injective and $\pi_1 \Gamma\rightarrow \pi_1B_r$ is injective.
\end{coro}

\subsection{Growth rates as eigenvalues of the non-backtracking matrix}\label{sub:growth_as_eigenvalue}

Let $\Gamma = (V,E)$ be a finite connected graph. 
% \begin{com}
% 	RC. Removed the definition of the adjacency matrix that was not used.
% \end{com}
%Let $A_\Gamma$ be its adjacency matrix:
% \[\{A\}_{u,v \in V^2}=\mathbf{1}\left( \{u,v\} \in E \right).\]

\begin{defn}[Non-backtracking matrix]\label{defn_of_non-backtracking}
    The \emph{non-backtracking matrix} of $\Gamma$ is the $\card E \times \card E$ matrix $M_\Gamma$ where the entry indexed by $(e,e')$ is one if $t(e) = o(e')$ and $e' \neq \bar e$ and zero otherwise.
 %    \begin{com}
 %    	RC. I changed the notation: $B_r$ is a graph, $M_\Gamma$ is a matrix.
	% This is too close for objects with different nature.
 %    \end{com}
\end{defn}

By the Perron--Frobenius Theorem, the leading eigenvalue $\lambda_1(M_\Gamma)$ of $M_\Gamma$ is a non-negative real number.
If $\Gamma$ is connected, then $\ln \lambda_1(M_\Gamma)$ is the exponential growth rate of balls in the universal cover of $\Gamma$ \cite{angel2015non}.
% \begin{com}DTW: is my adjustment ok? It was this:
% If $\Gamma$ is connected, then $\ln \lambda_1(M_\Gamma)$ equals the exponential of the growth rate  (in the log scale) of the cardinality of balls in the universal cover of $\Gamma$ \cite{angel2015non}.\\
% RC. This is good for me.
% \end{com}
This explains how $\lambda_1(M_\Gamma)$ relates to the growth of groups acting on a tree.
%In particular, $\lambda_1(M_\Gamma) \in (1, 2r-1)$ if and only if $\omega(\widetilde G) \in (0, \ln(2r-1))$.
%We study the growth rate $\omega(\widetilde G)$ by analyzing $\lambda_1(M_\Gamma)$.

\begin{theo}\label{thm:dense_spectral_radius_non-back}
    Let $r \in \N \setminus\{0, 1\}$, $\lambda \in (1, 2r-1)$ and $\epsilon > 0$. 
    There exists a finite connected graph $\Gamma = (V, E)$, whose degrees are in $\{2, \ldots, 2r\}$, such that 
    \begin{align}\label{|lambda-a|}
    	\abs{\lambda_1(M_\Gamma)- \lambda}\leq \epsilon.
    \end{align}
%    Equivalently, $|\omega(\widetilde G) - \omega_0| \leq \epsilon'$ for some $\epsilon' \to 0$ as $\epsilon \to 0$.
	% \begin{com}
	% 	RC. I don't think that it is necessary to rephrase each time in the of exponential growth rate. 
	% 	The explanation above should suffice.
	% \end{com}
\end{theo}

An important tool we use to prove \autoref{thm:dense_spectral_radius_non-back} is randomness. 
Specifically, we prove that a random graph with $n$ vertices satisfying a suitable vertex degree distribution satisfies \eqref{|lambda-a|} with probability tending to one as $n$ tends to infinity.
% \begin{com}
% 	RC. I tried to rephrase this, because the original sentence was unclear to me.
% \end{com}
%
%
%
%Specifically, we prove that for any length~$n$ graphic sequence $\mathbf K = \mathbf K(\alpha)$ satisfying some regularity assumptions, a uniformly chosen graph with degrees $\mathbf K$ satisfies \eqref{|lambda-a|} with probability tending to~1. 

\subsection{Probabilistic settings}
The random graph model we analyze is a random graph with a given degree sequence. 
We formally define the model.
For further details, see \cite[sec~1.4]{bordenave2016lecture}.
% \begin{com}
% 	RC. I found the notations difficult to follow: 
% 	e.g. $\mathbf d$ is a set of degrees, but $\mathbf d_n$ is a collection of cardinality;
% 	the elements of $\mathbf d$ are sometimes written $k$, sometimes $i$;
% 	sometimes its was written $d_n$ sometimes $\mathbf d_n$.
% 	So I proposed slightly different notations.
% 	Feel free to disagree!
% \end{com}

Let $2 \leq k_{\min} < k_{\max}$ be two positive integers, and let $\mathbf K=\{k_{\min},\ldots, k_{\max}\}$ be the set that will represent the possible degrees in the graph.
For each $n \in \N$, let $\mathbf N_n=(n_k)_{k \in \mathbf K}$ be a family with:
\begin{itemize}
    \item $\sum_{k \in \mathbf K} n_k=n$.
    \item $\sum_{k \in \mathbf K} kn_k$ is even.
\end{itemize}

\begin{defn}
\label{defn:configuration_model}
	Let $n \in \N$ and $\mathbf N_n = \{n_k\}_{k \in \mathbf K}$ be as above. 
	Define $\Gamma_n(\mathbf N_n)$ to be the uniform probability distribution over graphs with $n$ vertices, no loops, and with $n_k$ vertices of degree $k$, for every $k \in \mathbf K$. 
	We call $\mathbf N_n$ the \emph{degree distribution} of the random graph.
	% \begin{com}
	% 	RC. I slightly rephrased the definition that was ambiguous to me. 
	% 	Please check if I did not introduce any mistake.
	% 	How important is it really that the graph has no loop?
	% \end{com}
\end{defn}
% \begin{com}
% 	RC. Unless I am mistaken, for growth purpose we only use the "mutligraph" version of the model. 
% 	Indeed a graph with multiple edges or loops, is perfectly suitable to run the immersion argument.
% 	Serre's definition of graphs is essentially equivalent to multigraphs. 
% 	Hence I only kept this one.
% 	This also avoid the probably of writing the second model $\tilde G(d_n)$ that could be confused with the universal cover of $G(d_n)$.
% \end{com}

%
%
%In many cases it is easier to work with the more constructive model of random \emph{multigraphs} with a given degree sequence.
%
%\begin{defn}[Configuration model]\label{defn:configuration_model}
%Let $n \in \N$ and $d_n = \{n_i\}_{i \in \mathbf{d}}$ be as above. 
%Define $\Tilde{G}_n(d_n)$ to be the uniform probability distribution over multigraphs with $n$ vertices, and with degrees $d_n$. 
%\end{defn}

The configuration model can be described algorithmically as follows. 
Consider a set $V$ with $n$ elements with a partition
\begin{equation*}
	V = \bigsqcup_{k \in \mathbf K} V_k, 
	\quad \text{where} \quad 
	\card{V_k} = n_k, \ \forall k \in \mathbf K.
\end{equation*}
For each $k \in \mathbf K$ and $v \in V_k$, let 
% \begin{com}DTW: should this be $\Delta_k$ instead of $\Delta_v$ ?\\ RC. I believe $\Delta_v$ is correct. There is a set of half-edges for each vertex.\end{com} 
$\Delta_v^{(n)} = \set{(v,j)}{ 1 \leq j \leq k}$, and define $\Delta^{(n)}$ as the disjoint union of all $\Delta_v^{(n)}$.
One should think of $\Delta^{(n)}$ as the \emph{set of half-edges} of a graph $\Gamma$ with vertex set $V$ satisfying the constraints of the model.
The set of graphs with $n$ vertices, no loops and degree repartition is exactly the set of \emph{matchings} of $\Delta^{(n)}$, i.e.\ involutions without fixed point of $\Delta^{(n)}$.
Let $M(\Delta^{(n)})$ be the set of matchings of the $\Delta^{(n)}$. 
Given such a matching $\sigma$, there is an edge in the corresponding graph from $v$ to $v'$ if $\sigma$ sends $(v,j)$ to $(v',j')$.
%
%For a vertex $u$, let $d_u=\text{degree}(u)$, let $\Delta^{(n)}_u= \{(u,j), 1\leq j \leq d_u\} $ and let $\Delta^{(n)}=\cup_u \Delta^{(n)}_u$. 
%The set $\Delta^{(n)}$ is the \emph{set of half-edges} of $V(\Tilde{G}_n)$. 
%Then the set of multigraphs  with $n$ vertices and with degrees $d_n$ is exactly the set of matchings of the set $\Delta^{(n)}$, i.e.\ permutations of $\Delta^{(n)}$ that are their 
% own inverse and with no fixed points.  
%Let $M(\Delta^{(n)})$ be the set of matchings of the $\Delta^{(n)}$. 
%Given $\sigma \in M(\Delta^{(n)})$, two vertices are adjacent if two half-edges of those vertices are matched via $\sigma$. 
%
Instead of picking a graph uniformly at random, we can equivalently pick a matching in $M(\Delta^{(n)})$ uniformly at random. 
% Lastly we have:
% \begin{align*}
%     |M(\Delta^{(n)})|
%      \ = \ (|\Delta^{(n)}|-1)(|\Delta^{(n)}|-3)\cdots 1
%      \ = \ (|\Delta^{(n)}|-1){!}{!}
% \end{align*}
% where $k!!$ denotes $(k)(k-2)(k-4)\cdots$.
% \begin{com}
%     RC. Since we are not using it later, I am not sure what the computation $|M(\Delta^{(n)})|$ add to this section.
% \end{com}

\subsection{The probabilistic concentration result and density of growth rates}

%Denote by $\mathbf N = \{\mathbf N_n\}_{n \in \N}$ a sequence of degree sequences, by $\mathbf\Gamma= \{\Gamma_n(\mathbf N_n)\}_{n \in \N}$ a sequence of random graphs, where for each $n$ we draw $\Gamma_n$ from the uniform distribution over graphs with degrees $d_n$, and by $\tilde{\mathbf \Gamma}_n = \{\tilde \Gamma_n(d_n)\}$ a sequence of random multigraphs, drawn from the configuration model. 

The following is the key probabilistic result, whose proof appears in the companion paper \cite{louvaris2024density}.

\begin{theo}[{\cite[thm~4.6]{louvaris2024density}}]\label{thm:probabilistic_claim}
	Let $\mathbf N = \{\mathbf N_n\}_{n \in \N}$ be a sequence of degree distributions.
	Let $\mathbf\Gamma= \{\Gamma_n(\mathbf N_n)\}_{n \in \N}$ be a sequence of random graphs, where for each $n$ we draw $\Gamma_n$ from the uniform distribution over graphs  with degree distribution $\mathbf N_n$.
	Let $P$ be a probability measure on $\mathbf K$.
	
	Suppose that there is $C > 0$ such that for all sufficiently large $n$, the degree distribution $\mathbf N_n = (n_k)_{k \in \mathbf K}$ satisfies
	\begin{align}
	\label{rate of convergence assumption simple graph}
    		\abs{ \frac{n_k}n-P(k)} \leq \frac Cn, \quad \forall  k \in \mathbf K.
	\end{align}
	Then there exists $c > 0$ such that with high probability as $n$ tends to infinity
	\begin{equation*}
		\abs{\lambda_1 \left(M_{\Gamma_n}\right)-\lambda}  \leq   \frac 1{c \log n}, 
		\quad \text{where} \quad 
		\lambda =  \frac{\mathbb E[k(k-1)]}{\mathbb E[k]}.	
	\end{equation*}
	% \begin{com}
	% 	RC. I changed the formula for $\lambda$. 
	% 	I know what is the expectation of a random variable. 
	% 	I am less familiar with the expectation of a probability measure, or any function of this probability measure
	% \end{com}
\end{theo}

\begin{rema}
	In the statement, the expectations are computed with respect to the probability measure $P$ on $\mathbf K$.
\end{rema}

We can now prove \autoref{thm:dense_spectral_radius_non-back} using this probabilistic approach.

\begin{proof}[Proof of \autoref{thm:dense_spectral_radius_non-back}]
	Let $r \in \N$, $r \geq 2$, and let $\lambda \in (1, 2r-1)$. 
	We consider graphs whose vertices have degree $2$ or $2r$.
	In other words we choose a probability measure $P$ supported on $\{2, 2r\}$.
	Such a measure is determined by the value $x$ of $P(2)$.
	Observe that the function 
	\begin{equation*}
		x \mapsto  \frac{\mathbb E[k(k-1)]}{\mathbb E[k]} = \frac{ 2x + 2r(2r-1)(1-x)}{2x + 2r(1-x)} 
	\end{equation*}
	is continuous on $[0, 1]$, with range $[1, 2r-1]$,  so we can choose a probability measure $P$ with 
	\begin{equation*}
		\frac{\mathbb E[k(k-1)]}{\mathbb E[k]} = \lambda.
	\end{equation*}
Choose a sequence of degree distributions $\mathbf N_n=(n_k)_{k\in\{2,2r\}}$ satisfying the hypotheses of \autoref{thm:probabilistic_claim} for $P$.
	By \autoref{thm:probabilistic_claim}, for all sufficiently large $n$ there exists a graph $\Gamma$ with $n$ vertices and degrees in $\{2, 2r\}$ such that $\abs{\lambda_1(M_\Gamma) - \lambda} \leq \epsilon$.
	Note that if $\Gamma$ is disconnected, we can always replace it by a connected component on which the spectral radius of $M_\Gamma$ is realized.
    % \begin{com}
    %     DTW: Do we need to verify that it is connected? is that almost surely the case and we insist on that? 
    %     Or do we pass to a component like:
    %     "If $\Gamma$ is disconnected, replace it by a connected component on which the spectral radius of $M_\Gamma$ is realized." \\
    %     RC. Your paper does not discuss this part :-)
    %     It feels to me that the second option is the good one. \\
    %     RC. Looks good to me.
    % \end{com}
	\end{proof}
	
	We now prove the density of growth rates for the free group acting on its Cayley graph with respect to a free basis.
	
	\begin{proof}[Proof of the density part of \autoref{res: main new} for the free group]
	We identify $F_r$ with the fundamental group $\pi_1B_r$.
	We see each loop of $B_r$ as a generator, and write $X$ for the corresponding Cayley graph.
% old version
%Let $\omega \in (0, \ln(2r-1))$ and $\epsilon > 0$.	By \autoref{thm:dense_spectral_radius_non-back}, there is a graph $\Gamma$ whose vertex degrees are in $\{2, \dots, 2r\}$ and such that $\abs{\lambda(M_\Gamma) - e^\omega} \leq \epsilon$.	Since the logarithm is concave it yields, $\abs{\ln \lambda(M_\Gamma) - \omega} \leq e^{-\omega}\epsilon$.	By \autoref{cor:immersion_and_subtree},	there is an immersion $\Gamma \rightarrow B_r$, and so we can regard $H = \pi_1 \Gamma$ as a subgroup of $\pi_1B_r = F_r$.	With this identification, $\omega(H,X)$ is also the growth rate of balls in the universal cover $\tilde \Gamma$ of $\Gamma$, thus $\omega_H = \ln \lambda_1(M_\Gamma)$.	Consequently $\abs{\omega_H - \omega} \leq \epsilon$.
% \begin{com}DTW:
% Please double check my epsilon-management correction...
% \end{com}
% new version
Let $\omega \in (0, \ln(2r-1))$ and $\epsilon > 0$.
	Let
	$\eta = e^\omega(1-e^{-\epsilon})$.
	By~\autoref{thm:dense_spectral_radius_non-back}, there is a graph $\Gamma$ whose vertex degrees are in $\{2, \dots, 2r\}$ and such that
$\abs{\lambda_1(M_\Gamma)-e^\omega}\leq \eta.$
Since
$
e^\omega-\eta=e^{\omega-\epsilon}$
and %\quad\text{and}\quad
$e^\omega+\eta\leq e^{\omega+\epsilon},
$
we have
$\abs{\ln\lambda_1(M_\Gamma)-\omega}\leq \epsilon$.
By \autoref{cor:immersion_and_subtree}, there is an immersion $\Gamma \rightarrow B_r$, and so we can regard $H=\pi_1\Gamma$ as a subgroup of $\pi_1B_r=F_r$.
With this identification, $\omega(H,X)$ is the growth rate of balls in the universal cover $\widetilde \Gamma$, thus
$\omega_H=\ln\lambda_1(M_\Gamma)$.
Consequently $\abs{\omega_H-\omega}\leq\epsilon$.
\end{proof}

% -------------------------------------------------------
\section{Quasiconvex growth spectrum of free groups}
\label{sec:density_free_group}
% -------------------------------------------------------

In this section we prove the density part of \autoref{res: main new} for free groups with a convex-cocompact action on a hyperbolic space, using purely geometric arguments.
We start by recalling a classical lemma whose proof, relying on a simple counting argument, is left to the reader.

\begin{lemm}
\label{res: comparing growth}
	Let $(X_1, \distV[1])$ and $(X_2, \distV[2])$ be two metric spaces.
	Let $G$ be a group acting properly by isometries on both $X_1$ and $X_2$.
	Suppose that there are $(x_1, x_2) \in X_1 \times X_2$ and $(\kappa, \beta) \in \R_+^* \times \R_+$ such that for every $g \in G$,
	\begin{equation*}
	    \dist[2]{x_2}{gx_2} \leq \kappa \dist[1]{x_1}{gx_1} + \beta.
	\end{equation*}
	Then $\omega(G, X_1) \leq \kappa\, \omega(G, X_2)$.
\end{lemm}

\noindent
{\bfseries Word metric on the free group.}
Let $F$ be a finitely generated free group.
We fix once and for all a free basis $\{a_1, \dots, a_r\}$ of $F$.
We denote the corresponding word distance between $g_1, g_2 \in F$ by $\abs{g_1 - g_2}$.

\medskip\noindent
{\bfseries Convex-cocompact action.}
Let $(X,d)$ be a geodesic, $\delta$-hyperbolic metric space.
We fix a base point $o \in X$ and assume that $F$ acts properly by isometries on $X$.
%For simplicity we let $\omega_F = \omega(F,X)$.
We suppose that this action is convex-cocompact, i.e.\ the orbit map $F \to X$ sending $g$ to $go$ is a quasi-isometric embedding:
there are constants $(\kappa, \beta) \in \R_+^* \times \R_+$ such that
\begin{equation*}
    \kappa^{-1}\abs{g_1 - g_2} - \beta \leq \dist{g_1 o}{g_2 o} \leq \kappa \abs{g_1 - g_2}, \quad \forall g_1,g_2 \in F.
\end{equation*}

\medskip\noindent
{\bfseries The graph $\Gamma$.}
Let $N$ be an integer.
Let $\Gamma = (V,E)$ be a finite $2r$-regular graph whose girth is at least $N$.
Recall that $e \mapsto \bar e$ is the involution reversing the orientation of an edge $e \in E$, while $o(e)$ and $t(e)$ stand for the initial and terminal vertices of $e$ respectively.
We write $H$ for the fundamental group of $\Gamma$.

\medskip\noindent
{\bfseries Labellings and weights on $\Gamma$.}
A \emph{labelling} of $\Gamma$ is a map $w \colon E \to F \setminus\{1\}$ such that $w(\bar{e}) = w(e)^{-1}$, for every $e \in E$.
Such a labelling is \emph{reduced} if for every $e,e' \in E$ with $t(e) = o(e')$ and $e' \neq \bar{e}$, the word $w(e)w(e')$ is reduced.
Such a labelling induces a morphism $\phi \colon H \to F$, where the image of an element $h \in H$ is the concatenation of all the labels read on some (hence any) loop of $\Gamma$ representing $h$.
If the labelling is reduced, then $\phi$ is one-to-one.
Since $\Gamma$ is a finite $2r$-regular graph, there is a reduced labelling $w_0 \colon E \to F$ such that every label belongs to $\{a_1^{\pm 1}, \dots, a_r^{\pm 1}\}$ and the image of the corresponding map $\phi \colon H \to F$ has finite index in $F$.
In particular, $\omega(\phi(H), X) = \omega(F, X)$.

A \emph{weight} on $\Gamma$ is a map $k \colon E \to \N \setminus\{0\}$ such that $k(e) = k(\bar{e})$, for every $e \in E$.
Given such a weight we build a reduced labelling $w \colon E \to F \setminus\{1\}$ of $\Gamma$ as follows: $w(e) = w_0(e)^{k(e)}$, for every $e \in E$.
For simplicity, we write $H_k$ for the image of the corresponding morphism $H \to F$ and set $\omega_k = \omega(H_k, X)$.

\begin{prop}
\label{res: comparing growth for same change label}
	There is $C \in \R_+^*$, which does not depend on $\Gamma$, with the following property.
	Let $e_0 \in E$.
	Let $k, k' \colon E \to \N \setminus\{0\}$ be two weights on $\Gamma$.
	Suppose that
	\begin{equation*}
	    \abs{k'(e_0) - k(e_0)} = 1 \quad \text{and} \quad k'(e) = k(e), \ \forall e \in E \setminus\{e_0, \bar{e}_0\}.
	\end{equation*}
	Then $\abs{\omega_k - \omega_{k'}} \leq C/N$.
\end{prop}

\begin{proof}
	We denote by $\phi, \phi' \colon H \to F$ the monomorphisms associated to the labelling induced by the weights $k$ and $k'$ respectively.
	We claim that for every $h \in H$,
	\begin{equation*}
	    \dist o{\phi(h)o} \leq \left(1 + \frac{C_1}{N}\right)\dist o{\phi'(h)o} + C_2
	\end{equation*}
	where $C_1, C_2 \in \R_+^*$ do not depend on $h \in H$, nor on $\Gamma$.
	
	Let $h \in H$.
	As an element of the fundamental group of $\Gamma$ it can be represented by a non-backtracking oriented edge path $\gamma$ of length $n$.
	We decompose $\gamma$ as
	\begin{equation*}
	    \gamma = \gamma_0 f_1 \gamma_1 \dots f_p \gamma_p,
	\end{equation*}
	where $f_i$ is either $e_0$ or $\bar{e}_0$ and $\gamma_i$ is a path that does not cross $e_0$ or $\bar{e}_0$.
	If $f_i$ and $f_{i+1}$ have opposite orientation (\resp the same orientation), then $\gamma_i$ (\resp $f_i \cup \gamma_i$) is a non-trivial loop.
	Since $\Gamma$ has girth at least $N$, there are at least $N-1$ edges in $\gamma_i$ provided $i \notin\{0, p\}$, hence
	\begin{equation}
	\label{eqn: comparing growth for same change label - card}
	    p + (p-1)(N-1) \leq n \quad \text{i.e.} \quad p \leq \frac{n-1}{N} + 1.
	\end{equation}
	We set $g = \phi(h)$ and write $b_i$ (\resp $c_i$) for the element of $F$ labelling the edge $f_i$ (\resp the path $\gamma_i$).
	Note that $b_i = w_0(e_0)^{\pm k(e_0)}$.
	For $i \in \intvald{0}{p}$ let
	\begin{equation*}
	    x_{2i} = c_0 b_1 c_1 \cdots b_i o \quad \text{and} \quad x_{2i+1} = c_0 b_1 c_1 \cdots b_i c_i o,
	\end{equation*}
	with the convention that the empty word is trivial so that $x_0 = o$.
	% \begin{com}DTW:
	% The indexing here seems inconsistent at $i=0$, since the displayed formula for $x_{2i}$ does not match the convention $x_0=o$.
	% Maybe write the sequence as
	% \[
	% x_0=o,\quad x_1=c_0o,\quad x_2=c_0b_1o,\quad x_3=c_0b_1c_1o,\ldots
	% \]
	% and then give the general formula for $i\ge1$. \\
	% RC. It looks correct to me: $x_i$ uses the first $i$ letters of $c_0b1c_1\dots$.
	% Hence if $i = 0$, the prefix is empty and we get $o$.
	% This is what the convention below is supposed to clarify.
	% \end{com}
	By the Morse Lemma there is $D_0 \in \R_+^*$ (depending only on $\kappa$, $\beta$, $\delta$) and a sequence
	$y_0 = x_0, y_1, \dots, y_{2p+1} = x_{2p+1}$
	of points ordered in this way along $\geo o{go}$ and such that $\dist{x_i}{y_i} \leq D_0$.
	It follows from the triangle inequality that
	\begin{equation*}
	    \abs{\dist o{go} - \left(\sum_{i=0}^p \dist o{c_i o} + \sum_{i=1}^p \dist o{b_i o}\right)} \leq 4pD_0.
	\end{equation*}
	Proceeding similarly for $\phi'(h)$ and using $\dist{b_i o}{b'_i o} \leq \kappa$, we get
	\begin{equation}
	\label{eqn: comparing growth for same change label - diff}
	    \abs{\dist o{go} - \dist o{g'o}} \leq p D_1, \quad \text{where } D_1 = \kappa + 8D_0.
	\end{equation}
	% \begin{com}DTW:
	% Regarding the constant $D_1$.
	% We compare both $\dist o{go}$ and $\dist o{g'o}$ to corresponding sums, and each comparison seems to contribute an error of order $4pD_0$.
	% Is there a cancellation? Is the constant instead
	% something like $\kappa+8D_0$.
	% The value should be right if we declare it.\\
	% RC. Good catch. Fixed.
	% \end{com}
	Since the orbit map is a quasi-isometric embedding,
	\begin{equation}
	\label{eqn: comparing growth for same change label - n}
	    n \leq \kappa \min\left\{\dist o{go}, \dist o{g'o}\right\} + \kappa\beta.
	\end{equation}
	% \begin{com}DTW:
	% From
	% \[
	% \dist{o}{go}\ge \kappa^{-1}|g|-\beta
	% \]
	% one gets $|g|\le \kappa\dist{o}{go}+\kappa\beta$.
	% Since the reduced word length is at least $n$, the additive term might be $\kappa\beta$. \\
	% RC. fixed.
	% \end{com}
	Our claim now follows from \eqref{eqn: comparing growth for same change label - card}, \eqref{eqn: comparing growth for same change label - diff}, and \eqref{eqn: comparing growth for same change label - n}.
	According to \autoref{res: comparing growth}, we obtain
	\begin{equation*}
	    \omega_{k'} \leq \left(1 + \frac{\kappa D_1}{N}\right)\omega_k,
	\end{equation*}
	and since $\omega_k \leq \omega_F$, we get 
	\begin{equation*}
		\omega_{k'} - \omega_k \leq \frac{\kappa D_1 \omega_F}N
	\end{equation*}
	The other inequality is obtained by symmetry.
\end{proof}

\begin{lemm}
\label{res: all alpha achievable}
There is $C \in \R_+^*$, which does not depend on $\Gamma$, such that for every $\alpha \in \intval{0}{\omega_F}$, there is a weight $k$ on $\Gamma$ such that $\abs{\omega_k - \alpha} \leq C/N$.
\end{lemm}

\begin{proof}
	We denote by $C$ the constant given by \autoref{res: comparing growth for same change label}.
	Let $T$ be the universal cover of $\Gamma$ (endowed with the metric where each edge has length one) on which $H$ acts by isometries.
	Let $k \colon E \to \N\setminus\{0\}$ be the constant weight equal to one.
	Note that $\omega_k = \omega_F$.
	For every $m \in \N \setminus\{0\}$, denote by $mk$ the constant weight equal to $m$.
	We claim that $\omega_{mk}$ converges to $0$ as $m$ tends to infinity.
	Fix $m \in \N \setminus \{0\}$ and let $\phi \colon H \to F$ be the monomorphism associated to the labelling induced by $mk$.
	By construction there is a point $x \in T$ such that
	\begin{equation*}
		\abs{\phi(h) - 1}= m\dist[T]x{hx}, \quad \forall h \in H.
	\end{equation*}
	Since the orbit map $F \to X$ is a quasi-isometric embedding we get for every $h \in H$,
	\begin{equation*}
		m \dist[T]x{hx} \leq \kappa \dist[X] o{\phi(h)o} + \kappa\beta
	\end{equation*}
	It follows then from \autoref{res: comparing growth} that $\omega_{mk} \leq \kappa \omega(H,T) / m$, whence our claim.
	
	Let $\alpha \in \intval 0{\omega_F}$.
	Without loss of generality we can assume that $\alpha > 0$.
	In view of the above discussion, there is a weight $k_0$ such that $\omega_{k_0} \leq \alpha$.
	Starting from $k_0$, we can produce by induction a sequence of weights $k_0, k_1, k_2, \dots$ as follows: for $i  \in \N$, pick an edge $e_i \in E$ such that $k_i(e_i) \geq 2$ and define $k_{i+1}$ so that $k_i$ and $k_{i+1}$ coincides on $E \setminus \{ e_i, \bar e_i\}$ while $k_{i+1}(e_i) = k_i(e_i) - 1$ and  $k_{i+1}(\bar e_i) = k_i(\bar e_i) - 1$.
	After finitely many step we reach the weight $k_n = k$ which is constant, equal to one, and whose corresponding growth rate is $\omega_k = \omega_F$.
	Hence, there is $i \in \intvald 0{n-1}$ such that $\alpha$ lies in between $\omega_{k_i}$ and $\omega_{k_{i+1}}$.
	According to \autoref{res: comparing growth for same change label}, the distance between these two growth rates is at most $C/N$, whence the result.
\end{proof}

\medskip\noindent
{\bfseries Conclusion.}
Since $C$ does not depend on $\Gamma$, we can apply the construction to a graph $\Gamma$ with arbitrarily large girth.
This gives a way to approximate any $\alpha \in \intval{0}{\omega_F}$ with arbitrary precision by the growth rate of a finitely generated subgroup of $F$, completing the proof of the density part of \autoref{res: main new}.
\begin{rema*}
If the metric on $X$ is explicit (e.g.\ $X$ is the Cayley graph of $F$ with respect to a free basis), then the argument can be turned into an effective algorithm to compute a finitely generated subgroup whose growth rate approximates $\alpha$ with a prescribed precision.
\end{rema*}

% -------------------------------------------------------
\section{Full spectrum}
\label{sec:full_spectrum}
% -------------------------------------------------------

The density part of \autoref{res: main new} being established for free groups, we now prove that the growth spectrum is the full interval $[0, \omega_G]$.
The proof is inspired by the work of the last three authors in \cite{louvarisSubgroupsFreeGroup2025}.
Nevertheless, as explained in this section, the argument works for a general action.

\subsection{Hyperbolic geometry}

This part will need some more advanced features of hyperbolic geometry.
We start with a quick review of this matter.
We only give references for quantitative statements.
For a general introduction, we refer the reader to Gromov's original article \cite{Gromov:1987tk} or \cite{Coornaert:1990tj, Ghys:1990ki, Bridson:1999ky}.

\medskip\noindent
{\bfseries The four point inequality.}
Let $(X,d)$ be a proper, geodesic, metric space.
The Gromov product of three points $x,y,z \in X$ is
\begin{equation*}
    \gro xyz = \frac{1}{2}\left[\dist xz + \dist yz - \dist xy\right].
\end{equation*}
Let $\delta \in \R_+$.
For the remainder of this section we assume that $X$ is $\delta$-hyperbolic, i.e.\
\begin{equation}
\label{eqn: four point hyp}
    \min\left\{\gro xyt, \gro yzt\right\} \leq \gro xzt + \delta, \quad \forall x,y,z,t \in X.
\end{equation}

\medskip\noindent
{\bfseries Quasiconvex subsets.}
A \emph{projection} of a point $x \in X$ on a subset $Y \subset X$ is a point $y \in Y$ such that $\dist xy = d(x,Y)$.
Let $\alpha \in \R_+$.
A subset $Y \subset X$ is \emph{$\alpha$-quasiconvex} if for every $x \in X$, for every $y,y' \in Y$ we have $d(x,Y) \leq \gro y{y'}x + \alpha$.

\begin{lemm}[see for instance {\cite[lem~2.12]{Coulon:2014fr}}]
\label{res: proj qc}
	Let $Y$ be an $\alpha$-quasiconvex subset of $X$.
	Let $x,x' \in X$ and $p,p'$ be respective projections of $x$ and $x'$ on $Y$.
	Then the following holds:
	\begin{enumerate}
	    \item $\gro xyp \leq \alpha$, for every $y \in Y$;
	    \item $\dist p{p'} \leq \max\left\{\dist x{x'} - \dist xp - \dist x{p'} + 2\epsilon, \epsilon\right\}$, where $\epsilon = 2\alpha + \delta$.
	\end{enumerate}
\end{lemm}

In particular the projection on a closed $\alpha$-quasiconvex subset $Y$ is \emph{large-scale $1$-Lipschitz} in the sense that for every $x,x' \in X$, for every $p,p' \in Y$, respective projections of $x$ and $x'$ on $Y$ we have

\begin{equation*}
    \dist p{p'} \leq \dist x{x'} + 4\alpha + 2\delta.
\end{equation*}
Geodesics are $3\delta$-quasiconvex, see \cite[lem~2.24]{Coulon:2014fr}.
The $A$-neighborhood of an $\alpha$-quasiconvex subset is $2\delta$-quasiconvex whenever $A \geq \alpha$, see \cite[lem~2.13]{Coulon:2014fr}.

\medskip\noindent
{\bfseries Boundary at infinity.}
We denote by $\partial X$ the boundary at infinity of $X$ and set $\bar{X} = X \cup \partial X$.
The definition of the Gromov product $\gro xyz$ extends to triples where $x,y \in \bar{X}$ and $z \in X$.
Let $x \in X$ and $\xi \in \partial X$.
For every $L \in \R_+$, we let
\begin{equation*}
    V_\xi(x,L) = \set{z \in \bar{X}}{\gro \xi zx > L}.
\end{equation*}
As $L$ runs over $\R_+$, these sets form a neighborhood basis of $\xi$.
If $(x_n)$ and $(y_n)$ are two sequences converging to $x$ and $y$ respectively, the four point inequality yields
\begin{equation}
\label{eqn: estimate gromov product at infinity}
    \gro xyz \leq \liminf_{n \to \infty} \gro{x_n}{y_n}z \leq \limsup_{n \to \infty} \gro{x_n}{y_n}z \leq \gro xyz + 2\delta.
\end{equation}

\begin{lemm}
\label{res: proj vs gromov product}
	Let $\rho \colon \R_+ \to X$ be a geodesic ray from $x \in X$ to $\xi \in \partial X$.
	Let $y \in X$ and $p$ a projection of $y$ on $\rho$.
	Then $\abs{\dist xp - \gro \xi yx} \leq 7\delta$.
\end{lemm}

\begin{proof}
	For every $z \in X$ we have
	\begin{equation*}
	    \dist xp + \gro yzp = \gro zyx + \gro yxp + \gro xzp.
	\end{equation*}
	Passing to the limit as $z$ converges to $\xi$ and using \eqref{eqn: estimate gromov product at infinity}, we get
	\begin{equation*}
	    \abs{\dist xp - \gro \xi yx} \leq \max\left\{\gro y\xi p, \gro yxp + \gro x\xi p\right\} + 4\delta.
	\end{equation*}
	Since $p$ is a projection of $y$ on $\rho$, both $\gro y\xi p$ and $\gro yxp$ are bounded above by $3\delta$, while $\gro x\xi p = 0$.
\end{proof}

\begin{lemm}
\label{res: convex hull in neighborhood}
	Let $\xi \in \partial X$, $x \in X$ and $\ell \in \R_+$.
	Any point lying on a geodesic between two distinct points of $V_\xi(x, \ell)$ belongs to $V_\xi(x, \ell - 4\delta)$.
\end{lemm}

\begin{proof}
	Let $z$ and $z'$ be two distinct points of $V_\xi(x, \ell)$ and $y$ a point on a geodesic from $z$ to $z'$.
	The triangle inequality combined with \eqref{eqn: estimate gromov product at infinity} yields $\gro z{z'}x \leq \gro zyx + 2\delta$.
	Using the four point inequality we first get 
	\begin{equation*}
		\gro \xi yx 
		\geq \min \left\{ \gro \xi zx, \gro zyx \right\} - \delta
		\geq \min \left\{ \gro \xi zx, \gro z{z'}x \right\} - 3\delta
	\end{equation*}
	and then
	\begin{equation*}
		\gro \xi yx 
		\geq \min \left\{ \gro \xi zx, \gro \xi{z'}x \right\} - 4\delta
		 > \ell - 4\delta. \qedhere
	\end{equation*}
\end{proof}

\medskip\noindent
{\bfseries Group action.}
Let $G$ be a group acting properly by isometries on $X$.
An element $g \in G$ is either \emph{elliptic} (it has bounded orbits), \emph{parabolic} (it has exactly one accumulation point in $\partial X$), or \emph{hyperbolic} (it has exactly two accumulation points in $\partial X$).
The \emph{translation length} and \emph{stable translation length} are respectively
\begin{equation*}
    \norm g = \inf_{x \in X} \dist{gx}x \quad \text{and} \quad \snorm g = \lim_{n \to \infty} \frac{1}{n}\dist{g^n x}x.
\end{equation*}
Given $d \in \R_+$ we define the \emph{characteristic subset}
\begin{equation*}
    \fix{g,d} = \set{x \in X}{\dist x{gx} \leq d}.
\end{equation*}

\begin{lemm}
\label{res: fixed point set}
	Let $g \in G$.
	Let $d \geq 5\delta$ such that $\fix{g,d}$ is non-empty.
	The following holds.
	\begin{enumerate}
	    \item \label{enu: fixed point set - qc} The set $\fix{g,d}$ is $8\delta$-quasiconvex.
	    \item \label{enu: fixed point set - dist} $\dist x{gx} \geq 2d(x, \fix{g,d}) + d - 10\delta$, for every $x \in X \setminus \fix {g,d}$.
	    \item \label{enu: fixed point set - intersection} If $Y$ is a non-empty $\group g$-invariant, closed, $\alpha$-quasiconvex subset of $X$, then $\fix{g,d}$ intersects the $(\alpha + 5\delta)$-neighborhood of $Y$.
	\end{enumerate}
\end{lemm}

% \begin{com}DTW:
% 	Please check item~\ref{enu: fixed point set - dist} against the cited lemma.
% 	It seems the estimate cannot hold for arbitrary $d\ge 5\delta$.
% 	Indeed, in an $\mathbb R$-tree, let $g$ translate an axis by length $L$. 
% 	If $d>L$ and $x$ lies on the axis, then $\dist{x}{gx}=L$, while $d(x,\fix{g,d})=0$, so the displayed lower bound would imply $L\ge d$.
% 	Maybe the lower bound involves $\|g\|$, rather than $d$, or else $d$ should be fixed to a canonical value such as $d_0=\max\{\|g\|,5\delta\}$.
% 	This matters as \autoref{res: lower bound displacement} uses item~\ref{enu: fixed point set - dist}. \\
% 	RC. Fixed. The inequality is correct… provided $x$ is not already in the characteristic set.
% \end{com}

\begin{proof}
	Items~\ref{enu: fixed point set - qc} and \ref{enu: fixed point set - dist} are proved in \cite[lem~2.8]{Coulon:2018vp} when $d > \max\{\norm g, 5\delta\}$; the general case follows by continuity.
	% \begin{com}DTW:
	% I do not see how the ``general case follows by continuity'' for the estimate in item~\ref{enu: fixed point set - dist}, since the statement appears false when $d$ is larger than the translation length.
	% Should we quote the statement of \cite[lem~2.8]{Coulon:2018vp} here, or rewrite the lemma using the parameter that appears in that cited statement? \\
	% RC. Good catch. Fixed.
	% \end{com}
	For~\ref{enu: fixed point set - intersection}, consider a point $y \in Y$ and let $p$ be its projection onto $\fix{g,d}$.
	Then $\gro y{gy}p \leq 5\delta$ by~\ref{enu: fixed point set - dist}.
	It follows from the quasi-convexity of $Y$ that $p$ belongs to the $(\alpha + 5\delta)$-neighborhood of $Y$.
\end{proof}

An element $g \in G$ is hyperbolic if and only if $\snorm g > 0$.
In such a case the accumulation points of $g$ in $\partial X$ are
\begin{equation*}
    \xi_g = \lim_{n \to \infty} g^{-n}x \quad \text{and} \quad \xi_g' = \lim_{n \to \infty} g^n x.
\end{equation*}
If $\gamma \colon \R \to X$ is a bi-infinite geodesic from $\xi_g$ to $\xi_g'$, then
\begin{equation}
\label{eqn: translating biinfinite geodesics}
    \dist{g^n \gamma(t)}{\gamma(t + n\snorm g)} \leq 20\delta, \quad \forall n \in \Z, \ \forall t \in \R,
\end{equation}
see \cite[lem~2.11]{Coulon:2018vp}.

\begin{lemm}
\label{res: fixed point set large translation}
	Let $g \in G$ be hyperbolic with $\norm g > 8\delta$.
	Let $\gamma \colon \R \to X$ be a bi-infinite geodesic from $\xi_g$ to $\xi_g'$.
	For every $d \geq \norm g$, the set $\fix{g,d}$ is contained in the $A$-neighborhood of $\gamma$, where $A = \frac{1}{2}(d - \norm g) + 13\delta$.
\end{lemm}

\begin{proof}
	If $d > \norm g$, the statement is proved in \cite[lem~2.9]{Coulon:2018vp}.
	The general case follows by continuity.
\end{proof}

\begin{lemm}
\label{res: fixed point set elliptic}
	Let $g$ be a hyperbolic element of $G$.
	Let $\gamma \colon \R \to X$ be a bi-infinite geodesic from $\xi_g$ to $\xi_g'$.
	Let $u$ be an elliptic element which commutes with $g$.
	Then $\gamma$ is contained in $\fix{u, 25\delta}$.
\end{lemm}

\begin{proof}
	% \begin{com}DTW: should this be: 
    
 %    Let $x$ be an arbitrary point on $\gamma$. Then $\gro{\xi_g}{\xi_g'}x = 0$. \\
 %    RC. Works for me.
 %    \end{com}
    Let $x$ be an arbitrary point on $\gamma$. 
    Then $\gro{\xi_g}{\xi_g'}x = 0$.
	The set $Y = \fix{u, 5\delta}$ is $\group g$-invariant and $8\delta$-quasiconvex (\autoref{res: fixed point set}).
	By quasiconvexity, 
	\begin{equation*}
		d(x,Y) \leq \gro{g^{-n}y}{g^n y}x + 8\delta, \quad \forall y \in Y,\  \forall n \in \N.
	\end{equation*}
	Passing to the limit gives $d(x,Y) \leq 10\delta$, hence $x \in \fix{u, 25\delta}$.
\end{proof}

% \begin{com}DTW:
% 	This proves the conclusion for the chosen point $x\in\gamma$, but the lemma asserts  every point of $\gamma$ lies in $\fix{u,25\delta}$.
% 	Should the proof begin with an arbitrary $x\in\gamma$?
% 	If so, does $\gro{\xi_g}{\xi_g'}x=0$ holds for every point of the chosen bi-infinite geodesic under our conventions, or whether the constant $25\delta$ should be enlarged. \\
% 	RC. $\gro{\xi_g}{\xi_g'}x = 0$ is a consequence of being on $\gamma$.
% \end{com}

\begin{lemm}
\label{res: density fixed points}
	Let $U, V \subset \bar{X}$ be two open subsets intersecting $\Lambda(G)$.
	There is $g \in G$ such that $g(\partial X \setminus V) \subset U$.
\end{lemm}

\begin{proof}
	According to \cite[thm~2R]{Tukia:1994uk} there is a hyperbolic element $h \in G$ whose repelling and attracting points belong to $V$ and $U$ respectively.
	% \begin{com}DTW:
	% Are the attracting/repelling directions reversed?
	% A large positive power of $h$ sends points away from the repelling point and toward the attracting point.
	% Thus, to obtain $g(\partial X\setminus V)\subset U$, should the attracting point lie in $U$ and the repelling point lie in $V$?
	% Equivalently, if the fixed points are chosen as written, one needs a large negative power...\\
	% RC. Fixed.
	% \end{com}
	Since $h$ acts on $\bar{X}$ with a North-South dynamics, it suffices to take $g$ to be a sufficiently large power of $h$.
\end{proof}

\begin{lemm}
\label{res: lower bound displacement}
    Let $\rho \colon \R_+ \to X$ be a geodesic ray from $x \in X$ to $\xi \in \partial X$.
    Let $\ell \in \R_+$ and set $x_\ell = \rho(\ell)$.
%     \begin{com}DTW:
% Should instead  be: 
% ``Let $\ell > 4\delta$ and set $x_\ell = \rho(\ell)$.''
% The proof later uses $\gro \xi yx\ge \ell-4\delta$ to conclude that
% $x\notin Y=\fix{h,d}$. This holds when $\ell>4\delta$, and still supports later ping-pong application, where $\ell\ge\ell_0>28\delta$.
% \end{com}
    Let $h \in G$ be a hyperbolic isometry such that $\xi_h$ and $\xi_h'$ belong to $V_\xi(x, \ell + 13\delta)$.
    If $\dist{hx_\ell}{x_\ell} > 32\delta$, then
    \begin{equation*}
        \dist x{hx} \geq \dist x{x_\ell} + \dist{x_\ell}{hx_\ell} + \dist{hx_\ell}{hx} - 30\delta.
    \end{equation*}
\end{lemm}

\begin{proof} 
	Let $\gamma \colon \R \to X$ be a bi-infinite geodesic from $\xi_h$ to $\xi'_h$.
	For simplicity we set $Y = \fix{h,d}$, where $d = \max\{ \norm h, 5 \delta\}$.
	Let $y$ be a projection of $x$ on $Y$.
	We claim that $\gro \xi yx \geq \ell - 4\delta$.
	Suppose first that $\norm h > 8\delta$ so that $Y$ is contained in the $13\delta$-neighborhood of $\gamma$ (\autoref{res: fixed point set large translation}).
	In this case, our claim is a consequence of \autoref{res: convex hull in neighborhood}.
	Suppose now that $\norm h \leq 8\delta$.
	% \begin{com}DTW:
	% 	Could we include this four-point argument, or isolate it as a short lemma?
	% 	This claim is used in the estimate in the ping-pong proof, so one should see explicitly why the projection $y$ of $x$ to 
	% 	$\fix{h,\max\{\|h\|,5\delta\})}$ satisfies $\gro{\xi}{y}{x}\ge \ell-4\delta$ in the case $\|h\|\le 8\delta$. \\
	% 	RC. Expanded with the original proof.
	% \end{com}
	Assume also that our claim fails.
	Since the point $x_\ell$ lies on $\rho$, the product $\gro \xi x{x_\ell}$ vanishes.
	Applying the four point inequality, we get
	\begin{equation}
	\label{eqn: lower bound displacement}
		\min\left\{ \gro \xi{\xi_h}{x_\ell}, \gro {\xi_h}y{x_\ell}, \gro yx{x_\ell} \right\} \leq \gro \xi x{x_\ell} + 2\delta \leq 2\delta.
	\end{equation}
	On the one hand, it follows from the triangle inequality that
	\begin{equation*}
		\gro \xi{\xi_h}{x_\ell} \geq \gro \xi{\xi_h}x - \dist x{x_\ell} > 2 \delta.		
	\end{equation*}
	On the other hand, since we assume that our claim is false, the triangle inequality combined with (\ref{eqn: estimate gromov product at infinity}) yields
	\begin{equation*}
		\gro yx{x_\ell} \geq \dist x{x_\ell} - \gro \xi yx - \gro \xi x{x_\ell} - 2\delta > 2 \delta
	\end{equation*}
	Therefore the minimum in (\ref{eqn: lower bound displacement}) can only be achieved by the second term so that $\gro {\xi_h}y{x_\ell} \leq 2\delta$.
	The set $Y$ is $\group h$-invariant and $8\delta$-quasiconvex (\autoref{res: fixed point set}).
	Consequently $x_\ell$ belongs to the $12\delta$-neighborhood of $Y$.
	Hence $\dist {hx_\ell}{x_\ell} \leq d + 24\delta$.
	In particular, our assumption forces $\norm h > 8\delta$, a contradiction.
	
	Note that $\gro {\rho(t)} x{x_\ell} = 0$, whenever $t \geq \ell$.
	By the four point inequality -- see for instance \cite[lem~2.2(1)]{Coulon:2014fr} -- we get
	\begin{align*}
		\gro xy{x_\ell} & \leq \max \left\{ \dist x{x_\ell} - \gro {\rho(t)}yx, \gro {\rho(t)}x{x_\ell} \right\} + \delta, \\
		& \leq \max \left\{ \ell - \gro {\rho(t)}yx, 0 \right\} + \delta.
	\end{align*}
	Passing to the limit as $t$ tends to infinity, we get from our claim that $\gro xy{x_\ell} \leq 5\delta$.
	% \begin{com}DTW:
	% 	Please check the constant $5\delta$ here.
	% 	The preceding claim gives $\gro{\xi}{y}{x}\ge \ell-4\delta$, while \autoref{res: proj vs gromov product} has a $7\delta$ error.
	% 	Perhaps the constant should be enlarged? Is the implication obvious?\\
	% 	RC. Expanded with the original proof.
	% \end{com}
	Note that $x$ does not belong to $Y$ since $\gro \xi yx \geq \ell - 4\delta$.
According to \autoref{res: fixed point set}\ref{enu: fixed point set - dist}, since $y$ is a projection of $x$ on $Y=\fix{h,d}$, we have
% \begin{com}DTW:
% 	The previous displayed equation (commented out) had $\dist{hy}{hy}$, which is zero.
% 	We need the
% 	projection $y$ of $x$ onto $Y=\fix{h,d}$:
% 	\[
% 	\dist{hx}{x}\ge 2d(x,Y)+d-10\delta.
% 	\]
% 	Since $\dist{hx}{hy}=\dist{x}{y}=d(x,Y)$ and
% 	$\dist{hy}{y}\le d$, this gives the displayed inequality below. \\

%     RC. Correct. This was a typo.
% 	\end{com}
\begin{equation*}
		\dist {hx}{x} \geq \dist {hx}{hy} + \dist {hy}{y} + \dist{y}{x} - 10\delta.
\end{equation*}
% old version
%	According to \autoref{res: fixed point set}, 	\begin{equation*}		\dist {hx}x \geq \dist {hy}{hy} + \dist {hy}y + \dist yx - 10\delta.	\end{equation*}
	Together with the inequality $\gro xy{x_\ell} \leq 5\delta$, it yields
	\begin{equation*}
		\dist {hx}x \geq \dist {hx}{hx_\ell} + \dist{hx_\ell}{hy}  + \dist {hy}y + \dist y{x_\ell} + \dist {x_\ell}x - 30\delta.
	\end{equation*}
	The conclusion now follows from the triangle inequality.
\end{proof}

Given a subgroup $H$ of $G$, the \emph{limit set} of $H$, denoted $\Lambda(H)$, is the set of accumulation points in $\bar{X}$ of some (hence any) orbit of $H$.
The subgroup $H$ is \emph{elementary} if $\Lambda(H)$ contains at most two points.

\subsection{Semicontinuity of growth rates}

We now review some properties of growth rates related to conformal densities.
We refer the reader to Coornaert \cite{Coornaert:1993uv} and Coulon \cite{Coulon:2024pa}.

\medskip\noindent
{\bfseries Horocompactification.}
Although, there is a perfectly suitable notion of quasiconformal densities on the Gromov boundary, we prefer to work here with the horoboundary, where conformality is easier to state.
We assume that $X$ is a proper, geodesic, hyperbolic, metric space.
A \emph{cocycle} is a map $c \colon X \times X \to \R$ such that $c(x,z) = c(x,y) + c(y,z)$ for every $x,y,z \in X$.
We denote by $C^*(X)$ the set of continuous cocycles with the topology of uniform convergence on compact subsets of $X$.
We write $\iota \colon X \to C^*(X)$ for the map sending the point $z \in X$ to the cocycle $b_z$ defined as $b_z(x,y) =  \dist xz - \dist yz$, for all $x,y \in X$.
It induces a homeomorphism from $X$ onto its image.
The \emph{horocompactification} $\bar{X}_h$ of $X$ is the closure of $\iota(X)$ in $C^*(X)$.
The \emph{horoboundary} is $\partial_h X = \bar{X}_h \setminus \iota(X)$.
The action of $G$ on $X$ extends to a continuous action of $G$ on $\bar X_h$ preserving $\partial_hX$.
Moreover there is a natural $G$-equivariant surjective map from $\partial_h X \onto \partial X$ such that two cocycles $c$ and $c'$ have the same image if and only if they differ by a bounded cocycle.

\medskip\noindent
{\bfseries Conformal densities.}
Let $H$ be a group acting properly by isometries on $X$ and $\omega \in \R_+$.
A density is a collection $\nu = (\nu_x)_{x \in X}$ of positive measures on $\partial_hX$ such that $\nu_x \ll \nu_y$, for every $x, y \in X$.
Such a density is 
\begin{itemize}
	\item $\omega$-\emph{conformal}, if for every $x,y \in X$ we have
	\begin{equation*}
		\frac{d\nu_x}{d\nu_y} (c) = e^{-\omega c(x,y)}, \quad \nu_y\text{-almost everywhere.}
	\end{equation*}
	\item \emph{$H$-invariant}, if $h_\ast \nu_x = \nu_{hx}$, for every $h \in H$ and $x \in X$.
\end{itemize}
If $H$ is non-elementary, Patterson's construction provides an $H$-invariant, $\omega_H$-conformal density.
Conversely, any $H$-invariant, $\omega$-conformal density satisfies $\omega \geq \omega_H$.

\medskip\noindent
{\bfseries The space of subgroups.}
Given a countable group $G$, we denote by $\mathrm{Sub}(G)$ the set of all its subgroups, endowed with the Chabauty topology.

\begin{lemm}
\label{res: semicontinuity}
	Let $X$ be a proper geodesic hyperbolic metric space.
	Let $G$ be a group acting properly by isometries on $X$.
	Then the map $\mathrm{Sub}(G) \to \R_+$ sending $H$ to $\omega(H,X)$ is lower semicontinuous at every non-elementary subgroup.
\end{lemm}

\begin{proof}
	The argument follows the proof of \cite[thm~7.7]{McMullen:1999ha}.
	Consider a sequence $(H_k)$ of subgroups of $G$ converging to $H$.
	For every $k \in \N$, let $\nu^k = (\nu^k_x)$ be an $\omega_k$-conformal, $H_k$-invariant density, where $\omega_k = \omega(H_k, X)$.
	Without loss of generality, we can assume that $\nu^k$ is normalized to that the total mass of $\nu^k_o$ is $1$.
	Let $\omega = \liminf_{k \to \infty} \omega_k$.
	Up to passing to a subsequence, $\nu^k$ converges for the weak-* topology to an $\omega$-conformal density $\nu = (\nu_x)$.
	% \begin{com}DTW:
	% 	We should normalize the densities before taking a weak-* limit. Maybe require  $\nu^k_o(\partial_hX)=1$...?
	% 	Otherwise the limit could be zero, and the compactness statement is not immediate.\\
	% 	RC. Indeed, added a comment.
	% \end{com}
	For any $h \in H$, there exists $k_0$ such that $h \in H_k$ for all $k \geq k_0$, hence $h_* \nu_x = \nu_{hx}$.
	Thus $\nu$ is $H$-invariant, so $\omega \geq \omega_H$.
\end{proof}

\subsection{Application to growth spectra}

In this section, $X$ is a proper geodesic $\delta$-hyperbolic metric space and $G$ a group acting properly by isometries on $X$.
We write
\begin{equation*}
    \omega(\mathcal{H}, X) = \sup\set{\omega(H,X)}{H \in \mathcal{H}}.
\end{equation*}
As for growth rates, if there is no ambiguity, we simply denote this bound by $\omega_{\mathcal H}$.

\begin{defi}
\label{def: growth controlled}
A collection $\mathcal{H}$ of subgroups of $G$ is \emph{growth controlled} if for every $H, H' \in \mathcal{H}$ and every $\omega \in \R_+$ with $\omega > \max\{\omega(H,X), \omega(H',X)\}$, there exist $g \in G$ and a subgroup $M \in \mathcal{H}$ containing $\group{H, gH'g^{-1}}$ such that $\omega(M,X) < \omega$.
\end{defi}

\begin{rema}
\label{rem: growth control}
If $\mathcal{H}$ is invariant under conjugation, it is equivalent to ask for the existence of $M \in \mathcal{H}$ containing conjugates of $H$ and $H'$ with $\omega(M,X) < \omega$.
\end{rema}

We consider the following classes: $\mathcal{L}$ is the set of all subgroups of $G$ whose limit set is properly contained in $\Lambda(G)$, and $\mathcal{F}$ is the collection of all free, quasiconvex subgroups with infinite index in $G$.

\begin{prop}
\label{res: free product}
The classes $\mathcal{L}$ and $\mathcal{F}$ are growth controlled.
\end{prop}

\begin{proof}
	% \begin{com}
	% 	RC. Restored the original proof. Hopefully it adresses Dani's remarks.
	% \end{com}
	The proof is based on a classical ping-pong argument following \cite{Gitik:1999cx,Martinez-Pedroza:2009aa}.
	We start by defining some auxiliary objects.
	Denote by $F$ the maximal elliptic normal subgroup of $G$.
	Such a subgroup exists and is finite since the action of $G$ on $X$ is proper.
	In addition one can find a hyperbolic element $g \in G$ with $\norm g > 1000\delta$, such that the maximal elementary subgroup $E(g)$ of $G$ containing $g$ contains a subgroup $E^+(g)$ of index at most two in $E(g)$ and isomorphic to $F \rtimes \group g$.
	Indeed, although the action of $G$ is not cocompact, the construction goes as in \cite[lem~8]{Arzhantseva:2006cb}, see also \cite[lem~2.8]{Fujiwara:2024ra}.
	For simplicity we let $(\xi, \xi') = (\xi_g, \xi_g')$.
	We fix a bi-infinite geodesic $\gamma \colon \R \to X$ from $\xi$ to $\xi'$ and set $x = \gamma(0)$.
	
	Since $G$ is non-elementary, there is $\ell_0 > 28\delta$ such that $\Lambda(G)$ is \emph{not} contained in the closure of $V_\xi(x, \ell_0) \sqcup V_{\xi'}(x,\ell_0)$.
	Up to increasing the value of $\ell_0$ we can also assume that the following holds: for every $t,t' \in \R$ with $\abs{t'-t} \geq \ell_0$, for every element $u \in G$, if $u$ moves $\gamma(t)$ and $\gamma(t')$ by at most $500\delta$, then $u$ belongs to $F$.
	Indeed, if $\ell_0$ is sufficiently large (compare to $\norm g$) then such an element $u \in G$ must belong to $E^+(g)$.
	However elements of $F$ moves $\gamma$ by at most $25\delta$  (\autoref{res: fixed point set elliptic}), while $\norm g > 1000\delta$, whence the claim.
	
	\medskip
	Let us first prove the result for the class $\mathcal L$.
	Consider $H, H' \in \mathcal L$.
	The union of all geodesics joining two distinct points of $\Lambda(H)$ is an $H$-invariant, $6\delta$-quasi-convex subset.
	We write $Y$ for the set of points which are $20\delta$-close to such a geodesic. 
	It is closed, $2\delta$-quasiconvex, and $H$-invariant.
	% \begin{com}
	% 	RC. Detailed the construction of $Y$ to answer Dani's comment below.
	% \end{com}
	The set $Y'$ is defined form $\Lambda(H')$ in the same way.
	We denote by $Y$ and $Y'$ the sets of points which are $10\delta$-close to a bi-infinite geodesic joining two distinct points of $\Lambda(H)$ and $\Lambda(H')$ respectively.
	They are closed, $2\delta$-quasiconvex, and invariant under the action of $H$ and $H'$ respectively.
	Observe that $\mathcal L$ is invariant by conjugation.
	Moreover the growth rate of a subgroup is invariant under conjugation.
	According to \autoref{res: density fixed points}, we can assume, up to replacing $H$ and $H'$ by conjugates, that 
	\begin{equation*}
		Y\subset V_\xi(x, \ell_0)
		\quad \text{and} \quad
		Y' \subset V_{\xi'}(x, \ell_0).
	\end{equation*}
	The points $q$ and $q'$ stand for projections of $x$ on $Y$ and $Y'$ respectively.
	Note that $H$ is contained as a finite index subgroup in $H_0 = \group{H,F}$.
	In particular, $H_0$ and $H$ have the same growth rate and the same limit set.
	Moreover $H_0$ preserves $Y$.
	Similarly we let $H'_0 = \group{H',F}$.

	\medskip
	Let $\ell \in [\ell_0, \infty)$.
	Fix $\ell_1 > \ell + 33\delta$ and $\ell_2 > \ell_1 + \ell_0 + 7\delta$.
	Set
	\begin{equation*}
		y = \gamma(-\ell_1) \quad
		\quad \text{and} \quad 
		y' = \gamma(\ell_1).
	\end{equation*}
	There is $n \in \N$ such that
	\begin{equation*}
		g^{-n}Y \subset V_\xi(x, \ell_2), \quad
		g^nY' \subset V_{\xi'}(x, \ell_2)
		\quad \text{and} \quad 
		n \snorm g > \ell _1.
	\end{equation*}
	For simplicity we let $L = g^{-n} H_0 g^n$ and $L' = g^n H'_0 g^{-n}$.
	We are going to prove (among other things) the following facts: 
	\begin{enumerate}
		\item $L \cap L'=F$;
		\item the subgroup $M = \group {L,L'}$ is isomorphic to $L \ast_F L'$;
		\item the growth rate of $M$ is bounded above by a function of $\ell$ that converges to $\max\{\omega_H, \omega_{H'}\}$ as $\ell$ tends to infinity.
	\end{enumerate}

	\begin{clai}
		$L \cap L' =F$.
	\end{clai}
% \begin{com}DTW:
% There may be a constant issue here.
% \autoref{res: fixed point set}\ref{enu: fixed point set - intersection} gives 
% $\fix{u,5\delta}$ intersects a bounded neighborhood of $g^{-n}Y$.
% it does not give a point in $g^{-n}Y\cap\fix{u,5\delta}$.
% Likewise in Claim~\ref{cla: disp}.
% We either need to carry the nearby point and enlarge the constants below, or quote a stronger statement. \\
% RC. Fixed. $Y$ is already the neighborhood of an invariant quasi-convex subset. Hence there is no need to thicken again.
% Stressed this fact when we define $Y$.
% \end{com}
	\begin{proof}
		For simplicity we set $K = L \cap L'$.
		The inclusion $F \subset K$ follows from our construction.
		By construction $\Lambda(L)$ and $\Lambda(L')$ have an empty intersection.
		Hence $\Lambda(K) = \emptyset$, that is $K$ is elliptic.
		Let $u \in K$.
		By construction $g^{-n}Y$ is a suitable neighborhood an $L$-invariant, quasi-convex subset.
		According to \autoref{res: fixed point set}\ref{enu: fixed point set - intersection} there is a point $z \in g^{-n}Y \cap \fix{u, 5\delta}$.
		We denote by $p = \gamma(t)$ the projection of $z$ on $\gamma$.
		The points $z' \in g^nY'$ and $p' = \gamma(t')$ are defined in a similar way.
		It follows for our choice of $n$ combined with \autoref{res: proj vs gromov product} that $t \leq - \ell_2 + 7\delta$ and $t' \geq \ell_2 - 7\delta$.
		Hence, $\dist {p}{p'} > \max\{ \ell_0, 7\delta\}$.
		Since geodesics are $3\delta$-quasiconvex, we get $\gro {z}{z'}{p} \leq 7\delta$ and $\gro {z}{z'}{p'} \leq 7\delta$ (\autoref{res: proj qc}), so that  $p$ and $p'$ belongs to the $15\delta$-neighborhood of $ \fix{u, 5\delta}$ thus to to $\fix{u, 35\delta}$.
		According to our choice of $\ell_0$, it implies that $u$ is contained in $F$.
	\end{proof}

	\begin{clai}
	\label{cla: disp}
		For every $h \in L \setminus F$, we have $\dist {hy}{y} > 32\delta$.
	\end{clai}

	\begin{proof}
		Let $h \in L$ such that $\dist {hy}{y} \leq 32\delta$.
		In particular $\norm h \leq 32\delta$.
		As previously we observe that there is a point $z$ in $g^{-n}Y \cap \fix{h, 32\delta}$.
		Moreover its projection $p = \gamma(t)$ on $\gamma$ is such that $t \leq - \ell_2 + 7\delta$.
		In particular, $\dist py \geq \ell_0$.
		Since geodesics are $3\delta$-quasiconvex, we know that $\gro zyp \leq 3\delta$ (\autoref{res: proj qc}), hence $p$ and $y$ belong to the $11\delta$-neighborhood of $ \fix{h, 32\delta}$, and thus to $\fix{h, 54\delta}$.
		It follows from our choice $\ell_0$, that $h \in F$.
	\end{proof}
	
	\begin{clai}
	\label{cla: proj}
		Any projection of $y$ onto $g^{-n}Y$ is $5\delta$-close to $g^{-n}q$.
	\end{clai}
	
	\begin{proof}
		After translation by $g^n$ it suffices to prove that any projection $p$ of $g^ny$ on $Y$ is $5\delta$-close to $q$.
		Suppose on the contrary that it is not the case.
		The point $p$ and $q$ are respective projections of $g^ny$ and $x$ on $Y$.
		Since $Y$ is $2\delta$-quasiconvex, we get $\gro x{g^ny}q \leq 5\delta$ (\autoref{res: proj qc}).
		According to (\ref{eqn: translating biinfinite geodesics}) the point $g^ny$ is $20\delta$-close to $z = \gamma(-\ell_1+ T)$ where $T = n\snorm g$.
		Consequently,
		\begin{equation*}
			\gro xzq \leq \gro x{g^ny}q + 20\delta \leq 25\delta.
		\end{equation*}
		However geodesics are $3\delta$-quasiconvex, hence $q$ is $28\delta$-close to $\gamma$ restricted to $\intval 0{T-\ell _1}$.
		Hence $\gro \xi qx \leq 28\delta$.
		This contradicts the fact that as an element of $Y$ the point $q$ belongs to $V_\xi(x, \ell_0)$.
	\end{proof}

	\begin{rema}
		Claims analogous to Claims~\ref{cla: proj} and \ref{cla: disp} also holds for $L'$.
	\end{rema}

	\begin{clai}
	\label{cla: free product}
		Let $k \geq 1$.
%		Consider an element $h \in L$ of the form
		Consider a reduced alternating product
		\begin{equation*}
			h = h_1h_2\cdots h_k
		\end{equation*}
		where $h_1$, $h_2$, \dots, $h_k$ alternately belong to $L \setminus F$ and $L' \setminus F$.
		Then the following holds:
		\begin{enumerate}
			\item \label{enu: free product - neighborhood}
			$hx$ belongs to $V_\xi(x,\ell) \sqcup V_{\xi'}(x, \ell)$,
			\item \label{enu: free product - distance}
			$\displaystyle \dist x{hx} \geq \sum_{j = 1}^k \dist {h_jy_j}{y_j}  + 2k\ell$, \\ where $y_j$ is either $y$ or $y'$ depending whether $h_j$ belongs to $L$ or $L'$.
		\end{enumerate}
	\end{clai}
	
	\begin{proof}
		We define a sequence of points $x_0, x_1, \dots, x_k$ by letting $x_j = h_1\cdots h_jx$, with the convention that $x_0 = x$.
		In particular, $x_k = hx$.
		According to \autoref{res: lower bound displacement} we have
		\begin{equation}
		\label{eqn: free product - dist}
			\dist {x_j}{x_{j-1}} \geq \dist{h_jx}x > \dist{h_jy_j}{y_j} + 2 \ell  + 4\delta, \quad \forall j \in \intvald 1k,
		\end{equation}
		where $y_j$ is either $y$ or $y'$ depending whether $h_j$ belongs to $L$ or $L'$.
		\autoref{res: lower bound displacement} combined with the triangle inequality also yields
		\begin{equation}
		\label{eqn: free product - aux}
			\max\left\{ \gro{h_jx}x{y_j} , \gro{h_j^{-1}x}x{y_j} \right\}\leq 15\delta, \quad \forall j \in \intvald 1k.
		\end{equation}
		Fix $j \in \intvald 1{k-1}$.
		Because the $h_j$'s alternate between $L$ and $L'$, we have
		\begin{equation*}
			\gro{y_j}{y_{j+1}}x = \gro y{y'}x = 0.
		\end{equation*}
		Applying twice the four point inequality gives
		\begin{equation*}
			\min \left\{ \gro{y_j}{h_j^{-1}x}x, \gro{h_j^{-1}x}{h_{j+1}x}x, \gro{h_{j+1}x}{y_{j+1}}x \right\} \leq \gro{y_j}{y_{j+1}}x + 2\delta \leq 2\delta.
		\end{equation*}
    However, the minimum cannot be achieved by $\gro{y_j} {h_j^{-1}x}x$.
		Indeed in view of (\ref{eqn: free product - aux}) we have
		\begin{equation*}
			\gro{y_j}{h_j^{-1}x}x \geq \dist x{y_j} - \gro x{h_j^{-1}x}{y_j} \geq \ell_1 - 15\delta > 2\delta.
		\end{equation*}
		Similarly it cannot be achieved by $ \gro{h_{j+1}x}{y_{j+1}}x$.
		Hence
		\begin{equation*}
			\gro{h_j^{-1}x}{h_{j+1}x}x \leq  2\delta.
		\end{equation*}
		After translation this means that 
		\begin{equation}
		\label{eqn: free product - gromov}
			\gro{x_{j-1}}{x_{j+1}}{x_j} \leq 2\delta, \quad \forall j \in \intvald 1{k-1}.
		\end{equation}
		Point~\ref{enu: free product - distance} now follows from a proof by induction using (\ref{eqn: free product - dist}) and (\ref{eqn: free product - gromov}).
		Note that the sequence $x_0$, \dots, $x_k$ satisfies the assumption of \cite[lem~1]{Arzhantseva:2006cb}.
		Therefore $x_1 = h_1x$ is $16\delta$-close to any geodesic from $x$ to $hx$.
		Suppose that $y_1 = y$ (the other case is symmetric).
		By the triangle inequality, we have
		\begin{equation}
		\label{eqn: free product - final}
			\gro y\xi x \leq \gro {hx}\xi x + \gro{hx}x{h_1x} + \gro {h_1x}xy.
		\end{equation}
		According to the previous discussion, $\gro{hx}x{h_1x} \leq 16\delta$.
		Moreover by (\ref{eqn: free product - aux}) we have $ \gro {h_1x}xy \leq 15\delta$.
  %       \begin{com}DTW: this better?
%
  %       		It follows from (\ref{eqn: estimate gromov product at infinity}) that $\gro y \xi x \geq \ell_1 - 2\delta$.
		% Hence
		% $$
		% 	\gro {hx}\xi x
		% 	\geq \ell_1 - 2\delta - 16\delta - 15\delta
		% 	= \ell_1 - 33\delta
		% 	> \ell .
		% $$
		% Thus $hx$ belongs to $V_\xi(x,\ell)$, which completes the proof of \ref{enu: free product - neighborhood}. \\
%
  %       RC. Done.
  %       \end{com}
		It follows from (\ref{eqn: estimate gromov product at infinity}) that $\gro y \xi x \geq \ell_1 - 2\delta$.
        Plugin in these estimates in (\ref{eqn: free product - final}) we get 
        \begin{equation*}
            \gro {hx}\xi x
			\geq \ell_1 - 33\delta
			> \ell.
        \end{equation*}
		Thus $hx$ belongs to $ V_\xi(x, \ell)$ which completes the proof of \ref{enu: free product - neighborhood}.
	\end{proof}

	Recall that $F$ moves $x$ by at most $25\delta$ (\autoref{res: fixed point set elliptic}).
	\autoref{cla: free product}\ref{enu: free product - neighborhood} shows in particular that the element $h$ we considered does not belong to $F$.
	Since $F$ is normal in $G$, any element in $M$ can be written $hu$ with $h$ as in \autoref{cla: free product} and $u \in F$.
	It follows that $M$ is isomorphic to $L \ast_F L'$.
	Moreover the limit set of $M$ is contained in the closure of $V_\xi(x,\ell) \cup V_{\xi'}(x, \ell)$.
	It follows from our choice of $\ell_0$ that $\Lambda(M)$ is properly contained in $\Lambda(G)$, that is $M\in\mathcal L$.
	
	Let us now estimate the growth rate of $M$.
	As we observed, any element in $M$ can be written $hu$ with $h = h_1\cdots h_k$ as in \autoref{cla: free product} and $u \in F$.
	Note that this decomposition is not unique. 
	Consequently we will be over counting some elements in $M$.
	However this will not affect our final result.
	Following \autoref{cla: free product}\ref{enu: free product - distance} we get
	\begin{equation*}
		\dist {hux}x \geq \dist {hx}x - 25\delta \geq \sum_{j = 1}^k \dist{h_jy_j}{y_j} + 2k\ell - 25\delta
	\end{equation*}
	where $y_j$ is either $y$ or $y'$ depending whether $h_j$ belongs to $L$ or $L'$.
	Grouping the element of $L$ according to their syllable length, we get (with some over counting) the following estimate of the Poincaré series of $M$:
	\begin{equation}
	\label{eqn: free product}
		\mathcal P_M(s,x)
		\leq \card F e^{s(2\ell + 25\delta)} \sum_{k = 0}^\infty \left[ \mathcal P_L(s,y) \mathcal P_{L'}(s,y') e^{-4s\ell} \right]^k,
	\end{equation}
	compare for instance with \cite[Chapter~VI.A, prop~4]{Harpe:2000ab}.
%
% \begin{com}DTW: $L$ preserves $g^{-n}Y$, not $Y$... right? Is
%     this better?
%
%     	Let $p$ be a projection of $y$ onto $g^{-n}Y$, which is $2\delta$-quasiconvex and $L$-invariant.
% 	Since projection onto $g^{-n}Y$ is large scale $1$-Lipschitz (\autoref{res: proj qc}) we get
% 	\begin{equation*}
% 		\mathcal P_L(s,y)  \leq e^{10s\delta} \mathcal P_L(s,p).
% 	\end{equation*}
% 	However $p$ is $5\delta$-close to $g^{-n}q$ (\autoref{cla: proj}).
% 	The triangle inequality now yields
% 	\begin{equation*}
% 		\mathcal P_L(s,y)  
% 		\leq e^{20s\delta} \mathcal P_L(s,g^{-n}q) 
% 		=  e^{20s\delta} \mathcal P_{H_0}(s,q).
% 	\end{equation*}
%     RC. Fixed.
% \end{com}
    %
	Let $p$ be a projection of $y$ onto $g^{-n}Y$, which is $2\delta$-quasi-convex and $L$-invariant.
	Since projection onto $g^{-n}Y$ is large scale $1$-Lipschitz (\autoref{res: proj qc}) we get
	\begin{equation*}
		\mathcal P_L(s,y)  \leq e^{10s\delta} \mathcal P_L(s,p) 
	\end{equation*}
	However $p$ is $5\delta$-close to $g^{-n}q$ (\autoref{cla: proj}).
	The triangle inequality now yields
	\begin{equation*}
		\mathcal P_L(s,y)  
		\leq e^{20s\delta} \mathcal P_{g^{-n}H_0g^n}(s,g^{-n}q) 
		\leq  e^{20s\delta} \mathcal P_{H_0}(s,q).
	\end{equation*}
	The same argument gives
	\begin{equation*}
		\mathcal P_{L'}(s,y')
		\leq  e^{20s\delta} \mathcal P_{H'_0}(s,q').
	\end{equation*}
	In particular (\ref{eqn: free product}) becomes
	\begin{equation*}
		\mathcal P_M(s,x)
		\leq \card F e^{s(2\ell + 25\delta)} \sum_{k = 0}^\infty \left[ \mathcal P_{H_0}(s,q) \mathcal P_{H'_0}(s,q') e^{-4s(\ell- 10\delta)} \right]^k.
	\end{equation*}
	
	Consider now  $s, \omega \in \R_+$ with 
	\begin{equation*}
		\max \left\{ \omega(H, X), \omega(H', X)\right\} < s < \omega.
	\end{equation*}
	In particular the Poincaré series $\mathcal P_{H_0}(s,q)$ and $ \mathcal P_{H'_0}(s,q')$ are finite.
	Note also that the values of these series do not depend on $\ell$.
	Consequently if $\ell$ is sufficiently large then 
	\begin{equation*}
		\mathcal P_{H_0}(s,q) \mathcal P_{H'_0}(s,q') e^{-4s(\ell- 10\delta)} < 1
	\end{equation*}
	so that $\mathcal P_M(s,x)$ converges.
	Consequently $\omega(M,X) \leq s < \omega$.
 	We already observed that $M$ belongs to $\mathcal L$.
	By construction it contains conjugates of $H$ and $H'$.
	However $\mathcal L$ is invariant under conjugation. 
	Hence $\mathcal L$ is growth controlled (see \autoref{rem: growth control}).
	
	Let us now focus on the class $\mathcal F$.
	Suppose that $H,H' \in \mathcal F$, i.e.\ $H$ and $H'$ are free, quasiconvex with infinite index in $G$.
	In particular $H$ and $H'$ belong to $\mathcal L$.
	Thus we can follow the above construction.
	Consider the subgroup $M_0$ of $M$ generated by $g^{-n}Hg^n$ and $g^nH'g^{-n}$.
	By construction $M$ is generated by $M_0$ and $F$, while $F$ is normal in $M$.
	Consequently $M_0$ is a finite-index subgroup of $M$, hence with infinite index in $G$.
	Moreover $\omega(M_0, X) = \omega(M, X)$.
	Thus it suffices to prove that $M_0$ is free and quasiconvex.
	Since $H$ and $H'$ (and their conjugates) intersect $F$ trivially, \autoref{cla: free product}\ref{enu: free product - neighborhood}, tells us that $M_0$ is isomorphic to $H \ast H'$, whence a free group.
	Moreover, \autoref{cla: free product}\ref{enu: free product - distance} yields that the orbit map induces a quasi-isometric embedding of $M_0$ in $X$. 
	Therefore $M_0$ is quasiconvex.
\end{proof}

\begin{rema}
The same argument proves that the class of all quasiconvex (\resp quasiconvex and virtually free) subgroups of $G$ is growth controlled.
\end{rema}

We now highlight some properties of any growth controlled class of subgroups.

\begin{prop}
\label{res: sup realized}
	Let $\mathcal{H}$ be a growth controlled class of subgroups of $G$ containing a non-elementary subgroup and $\overline{\mathcal{H}}$ its closure in $\mathrm{Sub}(G)$.
	Then $\omega(\mathcal{H},X)$ belongs to $\Spec(\overline{\mathcal{H}},X)$.
\end{prop}

\begin{proof}
Write $\omega = \omega(\mathcal{H},X)$.
Without loss of generality, assume $\omega$ is not realized by any subgroup in $\mathcal{H}$.
By definition there is a sequence $(H_n)$ of elements in $\mathcal H$ such that $\omega(H_n, X)$ converges to $\omega$ from below as $n$ tends to infinity.
Without loss of generality, we can assume that $H_0$ is non-elementary.
We build by induction a non-decreasing sequence $(M_n)$ in $\mathcal{H}$.
First set $M_0 = H_0$.
Suppose now that $M_n$ has been defined.
Since $\max\{\omega(M_n,X), \omega(H_{n+1},X)\} < \omega$, growth control provides a group $M_{n+1} \in \mathcal{H}$ containing $\group{M_n, gH_{n+1}g^{-1}}$ for some $g \in G$.

Let $M$ be the union of all the groups $M_n$.
Then $\omega(H_n, X) \leq \omega(M_n, X) \leq \omega(M, X)$, so $\omega \leq \omega(M, X)$.
Observe that $(M_n)$ converges to $M$ in the Chabauty topology, hence $M$ belongs to $\overline{\mathcal{H}}$.
Since $M$ contains $H_0$, it is non-elementary.
Therefore \autoref{res: semicontinuity} gives $\omega(M, X) \leq \omega$.
\end{proof}
% \begin{com}DTW:
% \autoref{res: semicontinuity} is stated for non-elementary subgroups.
% Should  we justify that the limiting subgroup $M=\bigcup_n M_n$ is non-elementary, or should we split off the case $\omega=0$.
% As written, the application of \autoref{res: semicontinuity} is missing this hypothesis. \\
% RC. Good point. I added the assumption.
% \end{com}

Let $\omega \in \R_+^*$ and $\mathcal{H}$ a growth controlled class.
The subclass $\mathcal{H}_\omega = \{H \in \mathcal{H} : \omega(H,X) < \omega\}$ is also growth controlled, yielding:

\begin{coro}
\label{res: closure spectrum growth controlled class}
Let $\mathcal{H}$ be a growth controlled class of subgroups of $G$ consisting of non-elementary subgroups and $\overline{\mathcal{H}}$ its closure in $\mathrm{Sub}(G)$.
Denote by $\Omega$ the closure of $\Spec(\mathcal{H},X)$.
The interior of $\Omega$ is contained in $\Spec(\overline{\mathcal{H}},X)$.
\end{coro}

We complete this section with the proof of \autoref{res: free subgroup spectrum}.

% \begin{com}DTW:
% Maybe this proof should not cite \autoref{res: main new} as a black box, since the surface case below uses
% \autoref{res: free subgroup spectrum}.
% For Point~(1), we should instead apply the free-group density theorem to each convex-cocompact free subgroup and take the union over such subgroups.
% For Point~(2), we should apply \autoref{res: closure spectrum growth controlled class} to the non-elementary part of $\mathcal F$, and then handle the endpoint $0$ separately by cyclic/trivial subgroups. \\
% RC. Done.
% \end{com}
\begin{proof}[Proof of \autoref{res: free subgroup spectrum}]
Let $\epsilon \in \R_+^*$.
By definition, there is a non-abelian free quasi-convex subgroup $F \subset G$ such that $\omega_F \geq \omega_{\mathcal F} - \epsilon$.
Denote by $\mathcal F_0$ the collection of all non-abelian, finitely generated, infinite index subgroups of $F$.
It follows from the density part of  \autoref{res: main new} applied to the free group $F$ acting on $X$ that $\Spec(\mathcal F_0, X)$ is dense in $[0, \omega_F]$.
This fact holds for every $\epsilon \in \R_+^*$.
Hence $\Spec(\mathcal F, X)$ is dense in $[0, \omega_{\mathcal F}]$, which completes the proof of \ref{enu: free subgroup spectrum - density}.

Recall that  $\mathcal F$ is growth controlled (\autoref{res: free product}).
According to \autoref{res: closure spectrum growth controlled class} $(0, \omega_{\mathcal F})$ is contained in $\Spec(G,X)$.
Note that $0$ and $\omega_{\mathcal F}$ also belong to $\Spec(G,X)$.
Indeed $0$ is the growth rate of any cyclic subgroup of $G$.
The case of $\omega_{\mathcal F}$ is covered by \autoref{res: sup realized}.
Hence $[0, \omega_{\mathcal F}]$ is contained in $\Spec(G,X)$ as announced in \ref{enu: free subgroup spectrum - realized}.

We are left to prove \ref{enu: free subgroup spectrum - 1/2}, i.e.\ the inequality $\omega_{\mathcal{F}} \geq \omega_G / 2$.
Let $g \in G$ be a hyperbolic element and $N$ the normal closure of $g^n$ for a suitable large $n$.
By the theory of rotation families \cite{Dahmani:2017ef}, $N$ is a free subgroup, written as an ascending union of finitely generated quasiconvex subgroups $(H_k)$.
Each $H_k \in \mathcal{F}$, so $\omega(H_k, X) \leq \omega_{\mathcal{F}}$.
\autoref{res: semicontinuity} gives $\omega_N \leq \omega_{\mathcal{F}}$.
Since $N$ is an infinite normal subgroup of $G$, it is known that $\omega_N \geq \omega_G / 2$, with strict inequality when $G$ is divergent.
\end{proof}

\subsection{The surface case}

\begin{proof}[Proof of \autoref{res: main new} for surface groups]
Let $G$ be the fundamental group of a closed surface with a proper convex-cocompact action on a geodesic, hyperbolic, metric space $X$.
% \begin{com}
% 	DTW: difference between cocompact and convex-cocompact? \\
% 	RC. Fixed
% \end{com}
In view of \autoref{res: free subgroup spectrum}, it suffices to show that $\omega_{\mathcal{F}} = \omega_G$.
Let $N = [G,G]$ be the derived subgroup of $G$.
Since $G/N$ is abelian, $\omega_G = \omega_N$, see for instance \cite[thm~1.2]{Coulon:2024pa}.
Write $N$ as the union of an increasing sequence $(N_k)$ of finitely generated non-elementary subgroups.
By \autoref{res: semicontinuity}, $\omega(N_k, X)$ converges to $\omega_N$.
Since each $N_k$ is finitely generated with infinite index, it is a free quasiconvex subgroup, hence $\omega_{\mathcal{F}} = \omega_G$.
\end{proof}

\makebiblio{biblio}


@article{timar2026density,
  title={Density of growth rates of subgroups of a free group--an alternative proof},
  author={Tim{\'a}r, {\'A}d{\'a}m},
  journal={arXiv preprint arXiv:2601.12620},
  year={2026}
}

@article{louvaris2024density,
  title={Density of growth-rates of subgroups of a free group and the non-backtracking spectrum of the configuration model},
  author={Louvaris, Michail and Wise, Daniel T and Yehuda, Gal},
  journal={arXiv preprint arXiv:2404.07321},
  year={2024}
}

@book{Konig:1990th,
  title = {Theory of Finite and Infinite Graphs},
  author = {K{\"o}nig, D{\'e}nes},
  year = 1990,
  publisher = {Birkh\"auser Boston, Inc., Boston, MA},
  doi = {10.1007/978-1-4684-8971-2},
  isbn = {978-0-8176-3389-9},
  mrnumber = {1035708}
}

@misc{Coulon:2025cr,
  title = {{Croissance des sous-groupes du groupe libre, d'apr\`es Louvaris, Wise et Yehuda}},
  author = {Coulon, R{\'e}mi},
  year = 2025,
  month = nov,
  journal = {S\'eminaire virtuel francophone Groupes et G\'eom\'etrie},
  urldate = {2026-06-11},
  howpublished = {S\'eminaire virtuel francophone Groupes et G\'eom\'etrie, {https://plmbox.math.cnrs.fr/f/c31c0e8e10c548328a48}},
  langid = {french}
}

@misc{coulonTwistedPattersonSullivanMeasures2025,
  title = {Twisted {{Patterson-Sullivan Measures}} and {{Applications}} to {{Amenability}} and {{Coverings}}},
  author = {Coulon, R{\'e}mi and Dougall, Rhiannon and Schapira, Barbara and Tapie, Samuel},
  year = 2025,
  month = jan,
  journal = {American Mathematical Society},
  series = {Memoirs of the {{American Mathematical Society}}},
  volume = {305},
  number = {1539},
  publisher = {American Mathematical Society},
  issn = {0065-9266, 1947-6221},
  doi = {10.1090/memo/1539},
  urldate = {2025-01-30},
  howpublished = {https://www.ams.org/memo/1539/},
  isbn = {9781470480417 9781470470548},
  langid = {english}
}

@article{Arzhantseva:2006cb,
  title = {A Lower Bound on the Growth of Word Hyperbolic Groups},
  author = {Arzhantseva, Goulnara and Lysenok, Igor G.},
  year = 2006,
  month = feb,
  journal = {Journal of the London Mathematical Society},
  volume = {73},
  number = {1},
  pages = {109--125},
  publisher = {Oxford University Press (OUP)},
  doi = {10.1112/S002461070502257X},
  langid = {english},
  uri = {{$<$}a href="papers3://publication/doi/10.1112/S002461070502257X"{$>$}papers3://publication/doi/10.1112/S002461070502257X{$<$}/a{$>$}}
}

@book{Harpe:2000ab,
  title = {Topics in Geometric Group Theory},
  author = {{de la Harpe}, Pierre},
  year = 2000,
  series = {Chicago Lectures in Mathematics},
  publisher = {University of Chicago Press, Chicago, IL},
  isbn = {0-226-31719-6 0-226-31721-8},
  uri = {{$<$}a href="papers3://publication/uuid/B4E52818-7B9C-4B0C-B9CC-8734339F9F1F"{$>$}papers3://publication/uuid/B4E52818-7B9C-4B0C-B9CC-8734339F9F1F{$<$}/a{$>$}}
}

@article{Fujiwara:2024ra,
  title = {The Rates of Growth in an Acylindrically Hyperbolic Group},
  author = {Fujiwara, Koji},
  year = 2024,
  month = sep,
  journal = {Groups, Geometry, and Dynamics},
  volume = {19},
  number = {1},
  pages = {109--167},
  issn = {1661-7207},
  doi = {10.4171/ggd/820},
  urldate = {2025-02-28},
  langid = {english}
}

@misc{louvarisSubgroupsFreeGroup2025,
  title = {Subgroups of a Free Group with Every Growth Rate},
  author = {Louvaris, Michail and Wise, Daniel T. and Yehuda, Gal},
  year = 2025,
  month = may,
  number = {arXiv:2505.10650},
  eprint = {2505.10650},
  howpublished = {arXiv:2505.10650},
  primaryclass = {math},
  publisher = {arXiv},
  doi = {10.48550/arXiv.2505.10650},
  urldate = {2025-05-19},
  archiveprefix = {arXiv}
}

@book{Serre:1977wy,
  title = {Arbres, Amalgames, {{SL2}}},
  author = {Serre, Jean-Pierre},
  year = 1977,
  volume = {46},
  publisher = {Soci\'et\'e Math\'ematique de France, Paris},
  uri = {{$<$}a href="papers3://publication/uuid/EED907B0-C4BE-49BC-82FD-030AF1F45530"{$>$}papers3://publication/uuid/EED907B0-C4BE-49BC-82FD-030AF1F45530{$<$}/a{$>$}}
}

@article{McMullen:1999ha,
  title = {Hausdorff Dimension and Conformal Dynamics. {{I}}. {{Strong}} Convergence of {{Kleinian}} Groups},
  author = {McMullen, Curtis T.},
  year = 1999,
  month = jan,
  journal = {Journal of Differential Geometry},
  volume = {51},
  number = {3},
  publisher = {International Press of Boston},
  issn = {0022-040X},
  doi = {10.4310/jdg/1214425139},
  urldate = {2025-07-18}
}

@article{Coulon:2018aa,
  title = {Growth Gap in Hyperbolic Groups and Amenability},
  author = {Coulon, R{\'e}mi and Dal'bo, Fran{\c c}oise and Sambusetti, Andrea},
  year = 2018,
  journal = {Geometric and Functional Analysis},
  volume = {28},
  number = {5},
  pages = {1260--1320},
  publisher = {Springer International Publishing},
  doi = {10.1007/s00039-018-0459-6},
  langid = {english},
  uri = {{$<$}a href="papers3://publication/doi/10.1007/s00039-018-0459-6"{$>$}papers3://publication/doi/10.1007/s00039-018-0459-6{$<$}/a{$>$}}
}

@article{Corlette:1990br,
  title = {Hausdorff Dimensions of Limit Sets {{I}}},
  author = {Corlette, Kevin},
  year = 1990,
  month = dec,
  journal = {Inventiones Mathematicae},
  volume = {102},
  number = {1},
  pages = {521--541},
  doi = {10.1007/BF01233439},
  langid = {english},
  uri = {{$<$}a href="papers3://publication/doi/10.1007/BF01233439"{$>$}papers3://publication/doi/10.1007/BF01233439{$<$}/a{$>$}}
}

@article{Coornaert:1993uv,
  title = {Mesures de {{Patterson-Sullivan}} Sur Le Bord d'un Espace Hyperbolique Au Sens de {{Gromov}}},
  author = {Coornaert, Michel},
  year = 1993,
  journal = {Pacific Journal of Mathematics},
  volume = {159},
  number = {2},
  pages = {241--270},
  uri = {{$<$}a href="papers3://publication/uuid/3DCD4242-340D-48D8-89A7-51EF0B8F1B61"{$>$}papers3://publication/uuid/3DCD4242-340D-48D8-89A7-51EF0B8F1B61{$<$}/a{$>$}}
}

@article{Coulon:2024pa,
  title = {Patterson--{{Sullivan}} Theory for Groups with a Strongly Contracting Element},
  author = {Coulon, R{\'e}mi},
  year = 2024,
  month = mar,
  journal = {Ergodic Theory and Dynamical Systems},
  pages = {1--56},
  issn = {0143-3857, 1469-4417},
  doi = {10.1017/etds.2024.10},
  urldate = {2024-03-05},
  langid = {english}
}

@article{Dahmani:2017ef,
  title = {Hyperbolically Embedded Subgroups and Rotating Families in Groups Acting on Hyperbolic Spaces},
  author = {Dahmani, Fran{\c c}ois and Guirardel, Vincent and Osin, Denis V},
  year = 2017,
  month = jan,
  journal = {Memoirs of the American Mathematical Society},
  volume = {245},
  number = {1156},
  publisher = {American Mathematical Society},
  issn = {0065-9266},
  doi = {10.1090/memo/1156},
  isbn = {978-1-4704-2194-6},
  langid = {english},
  uri = {{$<$}a href="papers3://publication/doi/10.1090/memo/1156"{$>$}papers3://publication/doi/10.1090/memo/1156{$<$}/a{$>$}}
}

@article{Tukia:1994uk,
  title = {Convergence Groups and {{Gromov}}'s Metric Hyperbolic Spaces},
  author = {Tukia, Pekka},
  year = 1994,
  journal = {New Zealand Journal of Mathematics},
  volume = {23},
  number = {2},
  pages = {157--187},
  uri = {{$<$}a href="papers3://publication/uuid/4BD455F0-0F45-4AB5-9272-BCDB075E7030"{$>$}papers3://publication/uuid/4BD455F0-0F45-4AB5-9272-BCDB075E7030{$<$}/a{$>$}}
}

@book{Coornaert:1990tj,
  title = {G\'eom\'etrie et Th\'eorie Des Groupes},
  author = {Coornaert, Michel and Delzant, Thomas and Papadopoulos, Athanase},
  year = 1990,
  series = {Lecture Notes in Mathematics},
  volume = {1441},
  publisher = {Springer-Verlag, Berlin},
  isbn = {3-540-52977-2},
  uri = {{$<$}a href="papers3://publication/uuid/D1FF5E83-11D6-4D1C-ABE3-02F3D2C30E0F"{$>$}papers3://publication/uuid/D1FF5E83-11D6-4D1C-ABE3-02F3D2C30E0F{$<$}/a{$>$}}
}

@book{Ghys:1990ki,
  title = {Sur Les Groupes Hyperboliques d'apr\`es Mikhael Gromov},
  author = {Ghys, {\'E}tienne and {de la Harpe}, Pierre},
  editor = {{de la Harpe}, Pierre and Ghys, {\'E}tienne},
  year = 1990,
  series = {Progress in Mathematics},
  volume = {83},
  publisher = {Birkh\"auser Boston, Inc., Boston, MA},
  address = {Boston, MA},
  issn = {0743-1643},
  doi = {10.1007/978-1-4684-9167-8},
  isbn = {0-8176-3508-4},
  uri = {{$<$}a href="papers3://publication/doi/10.1007/978-1-4684-9167-8"{$>$}papers3://publication/doi/10.1007/978-1-4684-9167-8{$<$}/a{$>$}}
}

@book{Bridson:1999ky,
  title = {Metric Spaces of Non-Positive Curvature},
  author = {Bridson, Martin R. and Haefliger, Andr{\'e}},
  year = 1999,
  series = {Grundlehren Der Mathematischen Wissenschaften [{{Fundamental}} Principles of Mathematical Sciences]},
  volume = {319},
  publisher = {Springer-Verlag},
  address = {Berlin},
  issn = {3-540-64324-9},
  doi = {10.1007/978-3-662-12494-9},
  isbn = {3-540-64324-9},
  uri = {{$<$}a href="papers3://publication/doi/10.1007/978-3-662-12494-9"{$>$}papers3://publication/doi/10.1007/978-3-662-12494-9{$<$}/a{$>$}}
}

@incollection{Gromov:1987tk,
  title = {Hyperbolic Groups},
  booktitle = {Essays in Group Theory},
  author = {Gromov, Misha},
  year = 1987,
  pages = {75--263},
  publisher = {Springer, New York},
  address = {New York},
  uri = {{$<$}a href="papers3://publication/uuid/88DBB898-8049-4D49-AC8D-661BC612CA6D"{$>$}papers3://publication/uuid/88DBB898-8049-4D49-AC8D-661BC612CA6D{$<$}/a{$>$}}
}

@article{Coulon:2014fr,
  title = {On the Geometry of {{Burnside}} Quotients of Torsion Free Hyperbolic Groups},
  author = {Coulon, R{\'e}mi},
  year = 2014,
  journal = {International Journal of Algebra and Computation},
  volume = {24},
  number = {3},
  pages = {251--345},
  doi = {10.1142/S0218196714500143},
  langid = {english},
  uri = {{$<$}a href="papers3://publication/doi/10.1142/S0218196714500143"{$>$}papers3://publication/doi/10.1142/S0218196714500143{$<$}/a{$>$}}
}

@misc{Coulon:2018vp,
  title = {Infinite Periodic Groups of Even Exponents},
  author = {Coulon, R{\'e}mi},
  year = 2018,
  month = oct,
  eprint = {1810.08372v2},
  primaryclass = {math.GR},
  affiliation = {IRMAR},
  archiveprefix = {arXiv},
  howpublished = {arXiv 1810.08372},
  uri = {{$<$}a href="papers3://publication/uuid/42F68499-5BD2-4753-8DFA-A3685C512551"{$>$}papers3://publication/uuid/42F68499-5BD2-4753-8DFA-A3685C512551{$<$}/a{$>$}}
}

@article{Gitik:1999cx,
  title = {Ping-Pong on Negatively Curved Groups},
  author = {Gitik, Rita},
  year = 1999,
  journal = {Journal of Algebra},
  volume = {217},
  number = {1},
  pages = {65--72},
  doi = {10.1006/jabr.1998.7789},
  langid = {english},
  uri = {{$<$}a href="papers3://publication/doi/10.1006/jabr.1998.7789"{$>$}papers3://publication/doi/10.1006/jabr.1998.7789{$<$}/a{$>$}}
}

@article{Martinez-Pedroza:2009aa,
  title = {Combination of Quasiconvex Subgroups of Relatively Hyperbolic Groups},
  author = {{Mart{\'i}nez-Pedroza}, Eduardo},
  year = 2009,
  journal = {Groups, Geometry, and Dynamics},
  volume = {3},
  number = {2},
  pages = {317--342},
  issn = {1661-7207},
  doi = {10.4171/GGD/59},
  fjournal = {Groups, Geometry, and Dynamics},
  mrclass = {20F67 (20F65)},
  mrnumber = {2486802},
  mrreviewer = {Goulnara N. Arzhantseva}
}

@article{angel2015non,
  title={The non-backtracking spectrum of the universal cover of a graph},
  author={Angel, Omer and Friedman, Joel and Hoory, Shlomo},
  journal={Transactions of the American Mathematical Society},
  volume={367},
  number={6},
  pages={4287--4318},
  year={2015}
}

@article{bordenave2016lecture,
  title={Lecture notes on random graphs and probabilistic combinatorial optimization!! draft in construction!!},
  journal={in preparation},
  author={Bordenave, Charles},
  year={2016}
}

@article{LiWise2020,
    AUTHOR = {Li, Jiakai and Wise, Daniel T.},
     TITLE = {No growth-gaps for special cube complexes},
   JOURNAL = {Groups Geom. Dyn.},
  FJOURNAL = {Groups, Geometry, and Dynamics},
    VOLUME = {14},
      YEAR = {2020},
    NUMBER = {1},
     PAGES = {117--135},
      ISSN = {1661-7207},
   MRCLASS = {20F65 (20F69 57M05)},
  MRNUMBER = {4077657},
MRREVIEWER = {Michael Hull},
       DOI = {10.4171/ggd/537},
       URL = {https://doi-org.proxy3.library.mcgill.ca/10.4171/ggd/537},
}

@article{DahmaniFuterWise2019,
    AUTHOR = {Dahmani, Fran\c{c}ois and Futer, David and Wise, Daniel T.},
     TITLE = {Growth of quasiconvex subgroups},
   JOURNAL = {Math. Proc. Cambridge Philos. Soc.},
  FJOURNAL = {Mathematical Proceedings of the Cambridge Philosophical
              Society},
    VOLUME = {167},
      YEAR = {2019},
    NUMBER = {3},
     PAGES = {505--530},
      ISSN = {0305-0041},
   MRCLASS = {20F67 (20E08 20F69)},
  MRNUMBER = {4015648},
MRREVIEWER = {Enric Ventura Capell},
       DOI = {10.1017/s0305004118000440},
       URL = {https://doi-org.proxy3.library.mcgill.ca/10.1017/s0305004118000440},
}

@article {AngluinGardiner1981,
    AUTHOR = {Angluin, Dana and Gardiner, A.},
     TITLE = {Finite common coverings of pairs of regular graphs},
   JOURNAL = {J. Combin. Theory Ser. B},
  FJOURNAL = {Journal of Combinatorial Theory. Series B},
    VOLUME = {30},
      YEAR = {1981},
    NUMBER = {2},
     PAGES = {184--187},
      ISSN = {0095-8956},
   MRCLASS = {05C70},
  MRNUMBER = {615312},
MRREVIEWER = {W. D\"{o}rfler},
       DOI = {10.1016/0095-8956(81)90062-9},
       URL = {https://doi-org.proxy3.library.mcgill.ca/10.1016/0095-8956(81)90062-9},
}

@ARTICLE{Stallings83,
   author = {Stallings, John R.},
   title = {Topology of finite graphs},
   journal = {Invent. Math.},
   year = {1983},
   volume = {71},
   number = {3},
   pages = {551--565},
}

@article{dougall2016amenability,
  title={Amenability, critical exponents of subgroups and growth of closed geodesics},
  author={Dougall, Rhiannon and Sharp, Richard},
  journal={Mathematische Annalen},
  volume={365},
  number={3-4},
  pages={1359--1377},
  year={2016},
  publisher={Springer}
}
\end{document}